\documentclass{amsart}
\usepackage{frenchineq}
\usepackage{amsmath,amssymb,epsfig,amsthm}
\usepackage{color}
\usepackage{mathrsfs}
\usepackage{hyperref}
\usepackage[colorinlistoftodos,bordercolor=orange,backgroundcolor=orange!20,linecolor=orange,textsize=scriptsize]{todonotes}
\newcommand{\arxiv}[1]{\href{http://www.arXiv.org/abs/#1}{arXiv:#1}}

\DeclareMathAlphabet{\mathbbold}{U}{bbold}{m}{n}
\newcommand{\zero}{\mathbbold{0}}
\newcommand{\unit}{\mathbbold{1}}
\newcommand{\evalch}[1]{\eval{\ch}{}^{\min}_{#1}}
\newcommand{\proper}{^\dag}
\newcommand{\Psigmamin}{P_{\min}^\sigma}%
\newcommand{\Ptaumin}{P_{\min}^\tau}%
\newcommand{\Psigmainvmin}{P_{\min}^{\sigma^{-1}}}%
\newtheorem{theorem}{Theorem}[section]
\newtheorem{proposition}[theorem]{Proposition}
\newtheorem{lemma}[theorem]{Lemma}
\newtheorem{corollary}[theorem]{Corollary}
\theoremstyle{definition}
\newtheorem{definition}[theorem]{Definition}
\theoremstyle{remark}
\newtheorem{remark}[theorem]{Remark}
\newtheorem{example}[theorem]{Example}

\newcommand{\iY}{\mathsf{Y}}

\newcommand{\ptyP}{\mathscr{P}}
\def\d+{\dot{+}}
\def\card#1{\# #1}

\newcommand{\epi}{\operatorname{epi}}
\newcommand{\eps}{\epsilon}
\newcommand{\sC}{\mathcal{C}}

\newcommand{\corn}[1]{\mathsf{R} (#1)}
\newcommand{\cont}{\sC}

\newcommand{\dex}[1]{\mathsf{e}(#1)}
\newcommand{\dexo}{\mathsf{e}}
\newcommand{\Sym}{\mathfrak{S}}
\newcommand{\R}{\mathbb{R}}
\newcommand{\Rbar}{\overline{\R}}
\newcommand{\N}{\mathbb{N}}
\newcommand{\Z}{\mathbb{Z}}
\newcommand{\C}{\mathbb{C}}
\newcommand{\K}{\mathbb{K}}

\newcommand{\Ring}{\mathcal{R}}
\newcommand{\sA}{\mathcal{A}}
\newcommand{\sB}{\mathcal{B}}
\newcommand{\sF}{\mathcal{F}}
\newcommand{\sP}{\mathcal{P}}
\newcommand{\sL}{\mathcal{L}}
\newcommand{\sQ}{\mathcal{Q}}
\newcommand{\sM}{\mathcal{M}}

\newcommand{\sY}{\mathcal{Y}}
\newcommand{\sZ}{\mathcal{Z}}
\newcommand{\weakm}{\prec^{\rm w}}
\newcommand{\mrm}[1]{\text{\rm #1}}
\newcommand{\sat}{{\mrm{Sat}}}
\newcommand{\val}{\operatorname{val}}
\newcommand{\sgn}{\operatorname{sgn}}
\newcommand{\rmin}{\R_{\min}}
\newcommand{\rmax}{\R_{\max}}

\newcommand{\rhomin}{\rho_{\min}}
\newcommand{\eval}[1]{\widehat{#1}}
\newcommand{\divex}[1]{\overline{#1}}
\newcommand{\polfun}{\rmin\{\iY\}}
\newcommand{\vex}{\operatorname{vex}}
\newcommand{\perm}{\operatorname{per}}
\newcommand{\tr}{\operatorname{tr}}
\newcommand{\trm}{\tr^{\min}}
\newcommand{\ch}{\operatorname{ch}}
\newcommand{\diag}[1]{\operatorname{\mathsf{d}}(#1)}
\newcommand{\diagp}[1]{\operatorname{\mathsf{d}_\eps}(#1)}

\newcommand{\bydef}{\stackrel{\text{\rm def}}{=}}
\newcommand{\new}[1]{{\em #1}\index{#1}}
\newcommand{\set}[2]{\{#1\mid\,#2\}}

\newcommand{\opt}{\mrm{Opt}}
\newcommand{\idp}{\mrm{id}}
\newcommand{\id}{I}
\newcommand{\G}{\mathcal{G}}
\newcommand{\Perm}{\mathfrak{S}}

\title{Non-archimedean valuations of eigenvalues of matrix polynomials}
\author{Marianne Akian}
\address{Marianne Akian,
INRIA and CMAP, \'Ecole 
Polytechnique. Address:
CMAP, \'Ecole Polytechnique,
Route de Saclay,
91128 Palaiseau Cedex, France}
\email{Marianne.Akian@inria.fr}
\author{Ravindra Bapat}
\address{Ravindra Bapat, Indian Statistical Institute, New Delhi, 110016,
India} 
\email{rbb@isid1.isid.ac.in}
\author{St\'ephane Gaubert}
\address{St\'ephane Gaubert, 
INRIA and CMAP, \'Ecole 
Polytechnique. Address:
CMAP, \'Ecole Polytechnique,
Route de Saclay,
91128 Palaiseau Cedex, France}
\email{
Stephane.Gaubert@inria.fr}


\begin{document}
\maketitle

\begin{abstract}
We establish general weak majorization inequalities,
relating the leading exponents of the eigenvalues of matrices
or matrix polynomials over the field of Puiseux series
with the tropical analogues of eigenvalues.
We also show that these inequalities become equalities
under genericity conditions, and that the leading coefficients
of the eigenvalues are determined as the eigenvalues of auxiliary matrix
polynomials. 
\end{abstract} 

\sloppy

\section{Introduction}

\subsection{Non-archimedean valuations and tropical geometry}
A non-archimedean valuation $\nu$ on a field $\K$
is a map $\K\to \R\cup\{+\infty\}$ such that
\begin{subequations}\begin{gather}
\nu(a) = +\infty  \iff a=0\\
\nu(a+b)\geq \min(\nu(a),\nu(b))\\
\nu(ab)=\nu(a)+\nu(b) \enspace .
\end{gather}
\label{e-intro}
\end{subequations}
These properties imply that $\nu(a+b)=\min(\nu(a),\nu(b))$ 
for $a,b\in\K$ such that $\nu(a)\neq \nu(b)$. Therefore, the map $\nu$ 
is almost a morphism from $\K$ to the
{\em min-plus} or {\em tropical} semifield $\rmin$, which is
the set $\R\cup\{+\infty\}$, equipped
with the addition $(a,b)\mapsto \min(a,b)$
and the multiplication $(a,b)\mapsto a+b$.
A basic example of field with a non-archimedean valuation 
is the field of complex Puiseux series, with the valuation which
takes the leading (smallest) exponent of a series.
The images by a non-archimedean valuation of algebraic subsets of $\mathbb{K}^n$ are known as {\em non-archimedean amoebas}. The latter have
a combinatorial structure which is studied in tropical 
geometry~\cite{itenberg,sturmfelsmclagan}.
For instance, Kapranov's theorem
shows that the closure of the image by a non-archimedean valuation of an algebraic hypersurface over an algebraically closed field is a tropical hypersurface, i.e., the non-differentiability locus of a convex polyhedral function, see~{\cite{kapranov}}.
This generalizes 
the characterization of the leading exponents of the different branches
of an algebraic curve in terms of the slopes of the Newton polygon, 
which is part of the classical Newton-Puiseux theorem.

\subsection{Main results}
In the present paper, we consider the eigenproblem over the field of complex Puiseux series and related fields of functions.
Our aim is to relate the images of the eigenvalues by the non-archimedean
valuation with certain easily computable combinatorial objects called {\em tropical eigenvalues}.

The first main result of the present paper, Theorem~\ref{th-major-eig},
shows that the sequence of valuations of the eigenvalues
of a matrix
$\sA\in \K^{n\times n}$ 
is weakly
(super) majorized by the sequence of (algebraic) tropical
eigenvalues of the matrix obtained by applying the valuation to the
entries of $\sA$.
Next, we show
that the same majorization inequality holds under more general circumstances.
In particular, we consider in Theorem~\ref{th-major-eig-pre} 
and Corollary~\ref{cor-major-eig-pre} a relaxed definition
of the valuation, in the spirit of large deviations theory,
assuming that the entries of the matrix are functions
of a small parameter $\epsilon$. 
We do not require these functions to have a Puiseux
series type expansion, but assume that they have some mild form of first order
asymptotics. Moreover, the results apply to
a lower bound of the valuation of the entries of  $\sA$.

Then, in Section~\ref{sec-gen},
we assume that the entries of $\sA$ satisfy
\[ \sA_{ij}= a_{ij} \eps^{A_{ij}}+o(\eps^{A_{ij}})\enspace,
\]
for some scalars $a_{ij}\in \C$ and $A_{ij}\in \R\cup\{+\infty\}$,
as $\eps$ tends to $0$.
When $a_{ij}=0$, this reduces to $\sA_{ij}= o(\eps^{A_{ij}})$,
so that valuations are partially known: only a lower bound is known.
Applying Corollary~\ref{cor-major-eig-pre}, majorization inqualities 
are derived in Theorem~\ref{th-major-cont}.
The assumption of Section~\ref{sec-gen} is satisfied of course 
if the entries of $\sA$ are absolutely converging Puiseux series,
or more generally, if these entries belong to a polynomially bounded 
o-minimal structure~\cite{vandendries,alessandrini2013}.
We show in Theorem~\ref{th-gen=}
that the majorization inequalities of Theorem~\ref{th-major-cont}
become equalities for generic values of the entries $a_{ij}$.
The proof of the latter theorem relies
on some variations of the Newton-Puiseux
theorem, which we state as Theorems~\ref{th3-local} and~\ref{th3}.
The latter results only require a partial information on the asymptotics 
of the coefficients of the polynomial.
The particular case where this partial information 
contains at least the first order asymptotics of all the 
coefficients of the polynomial was considered in~\cite{dieu}.
However, here we show that a partial information on the first order asymptotics,
giving an outer approximation of a Newton polytope,
allow one to derive a partial information on the roots.
The latter idea goes back at least to~\cite{montel} in the
context of archimedean valuations. 

The valuation only gives an information on the {\em leading exponent}
of Puiseux series.
Our aim in Section~\ref{sec-pencils} is to refine this information,
by characterizing also the coefficients $\lambda_i\in \C$ of the leading monomials 
of the asymptotic expansions of the different eigenvalues
$\sL^i$ of the matrix $\sA$,
 \[ \sL^i\sim \lambda_i \eps^{\Lambda_i}, \qquad 1\leq i\leq n  \enspace .
 \]
As a byproduct, we shall end up with an explicit form, easily checkable,
of the genericity conditions under which the majorization inequalities
become equalities.
To this end, it is necessary to embed the
standard eigenproblem in the wider class of matrix
polynomial eigenproblems. 
Theorems~\ref{th-1} and~\ref{th-1-gen} show in particular that the 
coefficients $\lambda_i$ are the eigenvalues of certain auxiliary
matrix polynomials which are determined only by the leading exponents
and leading coefficients of the entries of $\sA$. These
polynomials are constructed from the optimal
dual variables of an optimal assignment problem,
arising from the evaluation of the tropical
analogue of the characteristic polynomial.

\subsection{Application to perturbation theory and discussion of related work}
The present results apply to perturbation theory~\cite{kato,baumgartel},
and specially, to the singular case
in which a matrix with multiple eigenvalues
is perturbed. The latter situation is the object of the theory developed 
by Vi\v sik and Ljusternik~\cite{vishik}
and completed by Lidski\u\i~\cite{lidskii},
see~\cite{mbo97} for a survey.
The goal
of this theory is to give a direct characterization
of the exponents, without computing
the Newton polytope of the characteristic polynomial.
The theorem of~\cite{lidskii} solves this problem under
some genericity assumptions, requiring the non-vanishing
of certain Schur complements. The question of solving
degenerate instance of Lidsk\u\i's theorem has been
considered in particular, by Ma and Edelman~\cite{maedelman} and Najman~\cite{najman}, and also by Moro, Burke and Overton in~\cite{mbo97}. Theorems~\ref{th-1} and~\ref{th-1-gen} 
generalize the theorem of Lidski\u\i, as they allow one
to solve degenerate instances in which the Schur complements
needed in Lidski\u\i's construction are no longer defined.

The present train of thoughts originates from a work of Friedland~\cite{friedland}, 
who showed that a certain deformation of the Perron root of a nonnegative matrix, in terms of Hadamard powers, converges to the maximal circuit mean of the matrix, a.k.a., the maximal tropical eigenvalue. 
Then, in~\cite{ABG96}, we showed that the limiting Perron eigenvector,
along the same deformation, can also be characterized by tropical means,
under a nondegeneracy condition. An early version of the present Theorems~\ref{th-major-cont} and~\ref{th-gen=} appeared as Theorem~3.8 of the authors' preprint~\cite{abg04b}. There, we also gave a generalization of the theorem of Lidski\u\i, in which the exponents of the first order asymptotics of the eigenvalues are given by the tropical eigenvalues of certain tropical Schur complements and their coefficients are given by some associated usual Schur complements like in the true Lidski\u\i\ theorem. 
However, some singular cases remained, 
see for instance Example~\ref{example-lidski},
motivating the introduction of matrix polynomial eigenproblems in further works.
The results of Theorems~\ref{th-1} and~\ref{th-1-gen} were announced without proof in the note~\cite{abg04}. 
Therefore, the present article is, for some part, a survey of results which have not appeared previously in the form of a journal article. It also provides a 
general presentation of eigenvalues in terms of valuation theory,
with several new results or refinements, like the general
majorization inequality for the eigenvalues of matrix polynomials,
Theorem~\ref{th-major-pencil}. 
The interest of this presentation is that it explains better the relation between the results obtained here, or in eigenvalues perturbation theory, for the non-archimedean valuations of matrix entries and eigenvalues, with their analogues for archimedean valuations, like the modulus map, or for
some generalization of the notion of  archimedean valuation which includes
in particular matrix norms.

Indeed, the latter works have motivated
a more recent work by Akian, Gaubert and Marchesini, who showed in~\cite{AGM2014a}, that one form of the theorem of Friedland concerning the Perron or dominant root carries over to all eigenvalues:  the sequence of moduli of all eigenvalues of a matrix is weakly log-majorized by the sequence of tropical eigenvalues, up to certain combinatorial coefficients. Therefore, the result there
can be thought of as a analogue of Theorem~\ref{th-major-eig} or~\ref{th-major-cont}
for the modulus archimedean valuation.
Also, the results of~\cite{abg04b,abg04} have
been at the origin of the application of tropical methods
by Gaubert and Sharify to the numerical computation and estimation
of eigenvalues~\cite{posta09,sharifyphd}, based on the norms of
matrix polynomial coefficients.
This is a subject of current interest,
with work by Akian, Bini, Gaubert, Hammarling, Noferini, Sharify, and Tisseur~\cite{binisharify,hamarling,SharifyTisseur,logmajorization2013}.

The present results provide a further illustration
of the role of tropical algebra in asymptotic analysis,
which was recognized by Maslov~\cite[Ch.~VIII]{maslov73}.
He observed that WKB-type or large deviation type asymptotics lead
to limiting equations, like Hamilton-Jacobi equations,
of a tropical nature.
This observation is at the origin
of idempotent analysis~\cite{maslov92,maslov92a,maslovkolokoltsov95,litvinov00}.
The same deformation has been identified by Viro~\cite{viro},
in relation with the patchworking method he developed
for real algebraic curves. 

Note that all the perturbation results described or recalled
above study sufficient conditions for the possible computation
of first order asymptotics of some roots when an information 
on the first order asymptotics of the data is only available.
However, when all the Puiseux series expansion of the data is known,
Puiseux theorem allows one to compute without any condition
all the Puiseux series expansion.
In the context of matrix polynomials, 
Murota~\cite{murota}, gave an algorithm to compute the Puiseux series 
expansions of the eigenvalues of a matrix polynomial depending polynomially
in the parameter $\eps$, avoiding the explicit computation of the
characteristic polynomial.
As for Theorems~\ref{th-1} and~\ref{th-1-gen}, 
his algorithm relies on a parametric optimal assignment problem.

The present work builds on tropical spectral theory. It has been inspired by the analogy with nonnegative matrix theory, of which Hans Schneider was a master. We gratefully acknowledge our debt to him.

\section{Min-plus polynomials and Newton polygons} \label{sec-minpol}
We first recall some elementary facts
concerning formal polynomials and polynomial functions
over the min-plus semifield, and their relation
with Newton polygons. 

The \new{min-plus semifield}, $\rmin$, is the set
$\R\cup\{+\infty\}$ equipped with the addition
$(a,b)\mapsto \min(a,b)$, denoted  $a\oplus b$, and the multiplication
$(a,b)\mapsto a+b$, denoted  $a\otimes b$ or $ab$.
We shall denote by $\zero=+\infty$
and $\unit=0$ the zero and unit elements of $\rmin$, respectively.
The familiar algebraic constructions and conventions carry
out to the min-plus context with obvious changes. For instance, if $A,B$ are
matrices of compatible dimensions
with entries in $\rmin$, we shall denote by $AB$ the matrix
product with entries $(AB)_{ij}
=\bigoplus_{k} A_{ik}B_{kj}
=\min_{k} (A_{ik}+B_{kj})$, we denote by $A^k$ the $k$th min-plus matrix power of $A$, etc.
Moreover, if $x\in \rmin\setminus\{\zero\}$, then
we will denote by $x^{-1}$ the inverse of $x$ for the $\otimes$ law, 
which is nothing but $-x$, with the conventional notation.
The reader
seeking information on the min-plus semifield
may consult~\cite{cuning,maslov92,bcoq,maslovkolokoltsov95,abg05,butkovicbook}.

We denote by $\rmin[\iY]$ the semiring of formal polynomials with 
coefficients in $\rmin$ in the indeterminate $\iY$: a
\new{formal polynomial} $P\in \rmin[\iY]$ is nothing but
a sequence $(P_k)_{k\in\N}\in \rmin^\N$ such that $P_k=\zero$
for all but finitely many values of $k$. Formal
polynomials are equipped with the entry-wise
sum, $(P\oplus Q)_k=P_k\oplus Q_k$, 
and the Cauchy product, $(P Q)_k=\bigoplus_{0\leq i\leq k}P_i Q_{k-i}$.
As usual, we denote a formal polynomial
$P$ as a formal sum, $P=\bigoplus_{k=0}^{\infty} P_k \iY^k$.
We also define the \new{degree} and \new{valuation} of $P$:
$\deg P=\sup\set{k\in \N}{P_k\neq \zero}$,
$\val P=\inf\set{k\in \N}{P_k\neq \zero}$
($\deg P=-\infty$ and $\val P=+\infty$ if $P=\zero$).
To any $P\in \rmin[\iY]$, 
we associate the \new{polynomial
function} $\eval{P}:\rmin\to\rmin,\; y\mapsto \eval{P}(y)=
\bigoplus_{k=0}^{\infty} P_k y^k$, that is, with the
usual notation:
\begin{align}
\eval{P}(y)=\min_{k\in \N} (P_k + k y) \enspace .
\label{e-minpoly}
\end{align}
Thus, $\eval{P}$ is concave, piecewise affine 
with integer slopes.
We denote by $\polfun$
the semiring of polynomial functions $\eval{P}$.
The morphism $\rmin[\iY]\to\polfun,\; P\mapsto \eval{P}$ is
not injective, as it is 
essentially a specialization
of the classical Fenchel transform over 
$\R$, which reads:
\[ \sF : \Rbar^{\R} \to \Rbar^{\R},\; \sF (g)(y)=\sup_{x\in\R} (xy-g(x))
\enspace,
\]
where $\Rbar:= \R\cup\{\pm\infty\}$.
Indeed, for all $y\in\R$, 
$\eval{P}(y)= -\sF(P)(-y)$, where the function
$k\mapsto P_k$ from $\N$ to $\rmin$ is extended to a function 
\begin{align}
\label{pasafunction}
P:\R\to \rmin,\, x\mapsto P(x),
\mrm{ with } P(x) = \begin{cases}
P_k &\mrm{if }x=k\in \N\enspace ,\\
+\infty &\mrm{otherwise}
\end{cases}
\end{align}

The following result
of Cuninghame-Green and Meijer gives
a min-plus analogue of the fundamental theorem of algebra.
\begin{theorem}[{\cite{cuning80}}]\label{th22}
Any polynomial function $\eval{P}\in \polfun$
can be factored in a unique way as
\begin{equation}
\label{th22.1} \eval{P}(y) = P_n (y\oplus c_1) \cdots 
(y\oplus c_n)\enspace ,
\end{equation}
with $c_1 \leq \cdots \leq c_n$.  
\end{theorem}
The $c_i$ will be called the \new{roots} of $\eval{P}$. 
\typeout{CHECK cuning80}
The \new{multiplicity} of the root $c$ is the cardinality of the set
 $\set{j\in\{1,\ldots, n\}}{c_j = c}$. 
We shall denote by $\corn{\eval P}$ the sequence of roots: 
$\corn{\eval{P}}=(c_1,\ldots,c_n)$.
By extension, if $P\in\rmin[\iY]$ is a formal polynomial,
we will call \new{roots} of $P$ the roots of $\eval{P}$,
so $\corn{P}:=\corn{\eval{P}}$. The next
properties also follow from~\cite{cuning80};
they show that the definition of the roots in Theorem~\ref{th22}
is a special case of the notion of a tropical hypersurface
defined as the nondifferentiability locus of a tropical polynomial~\cite{itenberg}.
\begin{proposition}[See~\cite{cuning80}]\label{prop-prenewton0}
The roots $c\in\R$ of a formal polynomial $P\in \rmin[\iY]$
are exactly the points at which the function $\eval{P}$ is not differentiable.
The multiplicity of a root $c\in\R$
is equal to the variation of slope
of $\eval{P}$ at $c$, $\eval{P}'(c^-)-\eval{P}'(c^+)$.
Moreover, $\zero$ is a root of $P$ if, and only if, $\eval{P}'(\zero^-):=
\lim_{c\to +\infty}\eval{P}'(c)\neq 0$.
In that case $\eval{P}'(\zero^-)$ is the multiplicity of $\zero$,
and it coincides with $\val P$. \hfill \qed
\end{proposition}
Legendre-Fenchel duality allows one to relate the tropical
roots to the slopes of Newton polygons. 
To see this, denote by $\vex f$ the convex
hull of a map $f:\R\to \Rbar$, and
denote by $\divex{P}$ the formal polynomial
whose sequence of coefficients is obtained by
restricting to $\N$ the map $\vex P$:
$k\mapsto \divex{P}_k:=(\vex P)(k)$, for $k\in \N$.
The function $\vex P$ is finite on
the interval $[\val P, \deg P]$.
Also, it should be noted that the graph of $\vex P$ is the standard
{\em Newton polygon} associated to the sequence of points $(k, P_k)$, $k\in [\val P,\deg P]$. 
\begin{theorem}[{\cite[Th.~3.43, 1 and 2]{bcoq}}]\label{th21}
A formal polynomial of degree $n$, $P \in \rmin[\iY]$, satisfies 
$P=\divex{P}$ if,
and only if, there exist $c_1\leq \cdots \leq c_n \in\rmin$ 
such that 
\[P = P_n (\iY\oplus c_1) \cdots (\iY\oplus c_n)\enspace
 .\]
The $c_i$ are unique and given,  by:
\begin{equation}\label{corners0}%
c_i=\begin{cases}  P_{n-i}  (P_{n-i+1})^{-1}& \mrm{if } P_{n-i+1}\neq\zero\\
\zero &\mrm{otherwise,}
\end{cases}\qquad \mrm{for } i=1,\ldots, n\enspace .
\end{equation}
\end{theorem}

The following standard observation
relates the tropical roots with the Newton polygon.
\begin{proposition}\label{prop-prenewton}%
The roots $c\in\R$ of a formal polynomial $P\in \rmin[\iY]$
coincide with the opposites of the slopes of the affine parts
of $\vex P:[\val P,\deg P]\to \R$.
The multiplicity of a root $c\in\R$
coincides with the length of 
the interval in which $\vex P$ has slope $-c$.
\hfill \qed 
\end{proposition}

\begin{remark}
The duality between tropical roots and slopes of the Newton polygon
in Proposition~\ref{prop-prenewton}
is a special case of the Legendre-Fenchel duality formula for
subdifferentials: $-c\in \partial (\vex P) (x)\Leftrightarrow
x\in \partial \sF (P) (-c) \Leftrightarrow
x\in \partial^+\eval{P} (c) $ where $\partial$ and $\partial^+$ denote
the subdifferential and superdifferential,
respectively~\cite[Th.~23.5]{rockafellar}.
\end{remark}
The above notions are illustrated in Figure~\ref{minpoly}, where
we consider the formal min-plus polynomial 
$P=\iY^3\oplus 5\iY^2 \oplus 6\iY \oplus 13$.
The map $j\mapsto P_j$, together with the 
map $\vex P$, are depicted at the left of the figure,
whereas the polynomial function $\eval{P}$ is depicted
at the right of the figure. We have $\divex{P}= \iY^3\oplus 3\iY^2 \oplus 6\iY\oplus 13
= (\iY\oplus 3)^2(\iY\oplus 7)$. Thus, the roots
of $P$ are $3$ and $7$, with respective multiplicities
$2$ and $1$. The roots are visualized at the right
of the figure, or alternatively, as the opposite
of the slopes of the two line segments at the
left of the figure. The multiplicities can be read
either on the map $\eval{P}$ at the right of the figure (the variation of slope
of $\eval{P}$ at points $3$ and $7$ is $2$ and $1$, respectively),
or on the map $\vex P$ at the left of the figure (as the respective horizontal
widths of the two segments). 
\begin{figure}[htb]
\begin{center}
\begin{picture}(0,0)%
\includegraphics{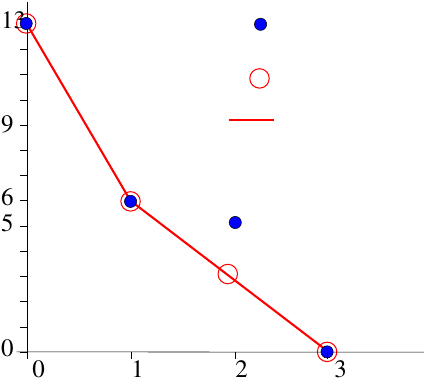}%
\end{picture}%
\setlength{\unitlength}{1316sp}%
\begingroup\makeatletter\ifx\SetFigFont\undefined%
\gdef\SetFigFont#1#2#3#4#5{%
  \reset@font\fontsize{#1}{#2pt}%
  \fontfamily{#3}\fontseries{#4}\fontshape{#5}%
  \selectfont}%
\fi\endgroup%
\begin{picture}(6117,5427)(811,-5476)
\put(4801,-526){\makebox(0,0)[lb]{\smash{{\SetFigFont{8}{9.6}{\rmdefault}{\mddefault}{\updefault}{\color[rgb]{0,0,0}$=P$}%
}}}}
\put(4816,-1239){\makebox(0,0)[lb]{\smash{{\SetFigFont{8}{9.6}{\rmdefault}{\mddefault}{\updefault}{\color[rgb]{0,0,0}$=\bar P$}%
}}}}
\put(4801,-1861){\makebox(0,0)[lb]{\smash{{\SetFigFont{8}{9.6}{\rmdefault}{\mddefault}{\updefault}{\color[rgb]{0,0,0}$=\vex P$}%
}}}}
\end{picture}%
\hskip 3em 
\begin{picture}(0,0)%
\includegraphics{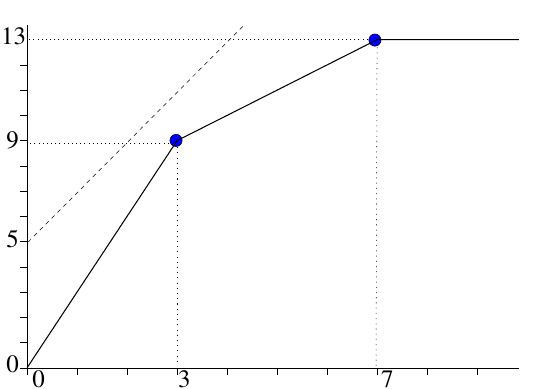}%
\end{picture}%
\setlength{\unitlength}{1316sp}%
\begingroup\makeatletter\ifx\SetFigFont\undefined%
\gdef\SetFigFont#1#2#3#4#5{%
  \reset@font\fontsize{#1}{#2pt}%
  \fontfamily{#3}\fontseries{#4}\fontshape{#5}%
  \selectfont}%
\fi\endgroup%
\begin{picture}(7951,5566)(811,-5401)
\put(2431,-3856){\makebox(0,0)[lb]{\smash{{\SetFigFont{8}{9.6}{\rmdefault}{\mddefault}{\updefault}{\color[rgb]{0,0,0}$3y$}%
}}}}
\put(4711,-1486){\makebox(0,0)[lb]{\smash{{\SetFigFont{8}{9.6}{\rmdefault}{\mddefault}{\updefault}{\color[rgb]{0,0,0}$6+y$}%
}}}}
\put(2401,-1261){\makebox(0,0)[lb]{\smash{{\SetFigFont{8}{9.6}{\rmdefault}{\mddefault}{\updefault}{\color[rgb]{0,0,0}$5+2y$}%
}}}}
\end{picture}%
\end{center}
\caption{The Newton poygon of the formal min-plus polynomial $P=\iY^3\oplus 5\iY^2 \oplus 6\iY \oplus 13$ (left) and the associated polynomial function $\eval P$ (right).\label{minpoly}}
\end{figure}
We conclude this section by two technical results. 
\begin{lemma}\label{carac-ci}%
Let $P=\bigoplus_{i=0}^n P_i \iY^i\in\rmin[\iY]$ be a formal polynomial
of degree $n$. Then, $\corn{P}=(c_1\leq \cdots \leq c_n)$ if, and only if,
$P\geq P_n (\iY\oplus c_1)\cdots (\iY\oplus c_n)$ and
\begin{equation}\label{carac-ci1}
 P_{n-i}=P_n c_1 \cdots c_i\quad \mrm{for all } i\in\{0,n\}\cup\set{i\in
\{1,\ldots , n-1\}}{c_i< c_{i+1}}\enspace .\end{equation}
In particular, $P_{n-i}=\divex{P}_{n-i}$ holds for
all $i$ as in~\eqref{carac-ci1}, and $\divex{P}= P_n (\iY\oplus c_1)\cdots (\iY\oplus c_n)$.
\end{lemma}
\begin{proof}
We first prove the ``only if'' part. 
If $\corn{P}=(c_1\leq \cdots \leq c_n)$,
then $\divex{P}=\divex{P}_n (\iY\oplus c_1)\cdots (\iY\oplus c_n)$ and,
by Theorem~\ref{th21},
$\divex{P}_{n-i}= \divex{P}_n c_1\cdots c_i$ for all $i=1,\ldots n$.
Recall that $P$ defines a map $x\mapsto P(x)$ by~\eqref{pasafunction}.
By definition of $\vex P$,
the epigraph of $\vex P$, $\epi \vex P$,
is the convex hull of the epigraph of $P$,
$\epi P$. By a classical result~\cite[Cor~18.3.1]{rockafellar}, if $S$ is a 
set with convex hull $C$, any extreme point of $C$ belongs
to $S$. Let us apply this to $S=\epi P$ and
$C=\epi \vex P$.
Since 
$c_i=\divex{P}_{n-i}  (\divex{P}_{n-i+1})^{-1}$,
the piecewise affine map $\vex P$
changes its slope at any 
point $n-i$ such that 
$c_i<c_{i+1}$.
Thus, any point $(n-i,\vex{P}(n-i))$ with $c_i<c_{i+1}$
is an extreme point of $\epi \vex P$,
which implies that $(n-i,\vex{P}(n-i))\in \epi P$,
i.e., $P_{n-i}\leq \vex{P}(n-i)=\divex P_{n-i}$. Since
the other inequality is trivial by definition of the convex hull, we
have $P_{n-i}=\divex{P}_{n-i}$. 
Obviously, $P$ and $\divex P$ have the same
degree, which is equal to $n$, and they have the same
valuation, $k$. Then, $(n,\vex P(n))$
and $(k,\vex P(k))$ are extreme points
of $\epi \vex P$, and by the preceding argument,
$P_n=\divex P_n$, and $P_k=\divex P_k$.
Hence, $P_0=\divex P_0$, if $k=0$,
and $P_0=\divex P_0=+\infty$, if $k>0$.
We have shown ~\eqref{carac-ci1},
together with the last statement of the lemma.
Since $\divex{P}_n=P_n$ and $P\geq \divex{P}$, we also obtain 
$P\geq P_n (\iY\oplus c_1) \cdots (\iY\oplus c_n)$.

For the ``if'' part, assume that $P\geq P_n (\iY\oplus c_1) \cdots (\iY\oplus c_n)$
and that~\eqref{carac-ci1} holds. Since $Q=P_n (\iY\oplus c_1) \cdots
 (\iY\oplus c_n)$ is convex, and the convex hull map 
$P\mapsto \divex{P}$ is monotone, we must have $\divex{P}\geq \divex{Q}=Q$.
Hence, $ P \geq \divex{P}\geq Q$ and since $P_{n-i}=Q_{n-i}$ 
for all $i$ as in~\eqref{carac-ci1}, we must have $P_{n-i}=\divex{P}_{n-i}=
Q_{n-i}$, thus $\vex P(n-i)=Q(n-i)$
at these $i$. Since $\vex P$ and $Q$ are convex, $Q$ is piecewise
affine and $Q(j)=\vex{P}(j)$ for $j$ at the boundary of the domain
of $Q$ and at all the $j$ where $Q$ changes of slope, we must
have $\vex{P}=Q$. Hence $\divex{P}=Q$
and $\corn{P}=\corn{\divex{P}}=\corn{Q}=(c_1,\ldots , c_n)$.
\end{proof}

\begin{corollary}\label{carac-mul}%
Let $P=\bigoplus_{i=0}^n P_i \iY^i\in\rmin[\iY]$ be a formal polynomial
of degree $n$. Let $c\in\R$ be a finite root of $P$ with multiplicity
$m$, and denote by $m'$  the sum of the multiplicities of all
the roots of $P$ greater than $c$ ($+\infty$ comprised).
Then, $P_{i}=\divex{P}_{i}$ for both $i=m'$  and $i=m+m'$, 
$\eval{P}(c)=P_{m'} c^{m'}=P_{m+m'}c^{m+m'}$ and
$\eval{P}(c)<P_{i}c^{i}$  for all $1\leq i<m'$ and $m+m'<i\leq n$.
\end{corollary}
\begin{proof}
Let us denote $\corn{P}=(c_1\leq \cdots \leq c_n)$.
By definition of $c$, $m$ and $m'$ we have $m\geq 1$, $m'\geq 0$,
$m+m'\leq n$,  $c=c_{n-m'-m+1}=\cdots=c_{n-m'}$, 
$c_{n-m'-m}<c$ if $n-m'-m>0$, and $c<c_{n-m'+1}$ if $n-m'<n$.
By Lemma~\ref{carac-ci1}, this implies that
for both $i=m'$  and $i=m+m'$,
$P_{i}=\divex{P}_{i}=P_n c_1 \cdots c_{n-i}$.
We also have
Since $\divex{P}_i = (\vex P)(i) \leq P_i $,
we have
$P_{i}\geq (\vex P)(i) =  \divex{P}_{i}=P_n c_1 \cdots c_{n-i}$ for all $i=0,\ldots, n$.
Moreover, by Theorem~\ref{th22}, we have $\eval{P}(c)=P_n 
(c\oplus c_1)\cdots (c\oplus c_n)= P_n c_1\cdots c_{n-m'-m} c^{m+m'}
=P_n c_1\cdots c_{n-m'} c^{m'}$.
Hence,  $\eval{P}(c)=P_{m+m'}c^{m+m'}=P_{m'} c^{m'}$,
and $\eval{P}(c)<P_n c_1 \cdots c_{n-i} c^i\leq P_i c^i$ for $i<m'$ and for
$i>m+m'$.
\end{proof}

\section{Tropical eigenvalues}
\label{sec-tropeig}
We now recall some classical results on tropical eigenvalues
and characteristic polynomials.

The {\em permanent} of a matrix with coefficients in an arbitrary semiring $(S,\oplus,\otimes)$ is defined by
\begin{align*}
\perm (A)&= \bigoplus_{\sigma\in \Sym_n}
\bigotimes_{i=1}^n A_{i\sigma (i)} \enspace,
\end{align*}
where $\Sym_n$ is the set of permutations of $[n]:=\{1,\ldots, n\}$.
In particular, for any matrix $A\in \rmin^{n\times n}$, 
\begin{align*}
\perm (A) =\min_{\sigma\in \Sym_n} |\sigma|_A \enspace ,
\end{align*}
where for any permutation $\sigma\in\Sym_n$,
we define the {\em weight} of $\sigma$ with respect to $A$ as
$|\sigma|_A:= A_{1\sigma(1)}+\cdots +A_{n\sigma(n)}$.

To any min-plus $n\times n$ matrix $A$, 
we associate the {\em (directed) graph} $G(A)$,
 which has set of nodes $[n]$ and an arc $(i, j)$
if $A_{ij}\neq \zero$, and the weight function which
associates the  weight $A_{ij}$ to the arc $(i, j)$ of $G(A)$.
In the sequel,
we shall omit the word ``directed'' as all graphs will be directed.
Then, $\perm (A)$ is the value of an optimal
assignment in this weighted graph.
It can be computed in $O(n^3)$ time using the Hungarian
algorithm~\cite[\S~17]{schrijver}.
We refer the reader to~\cite[\S~2.4]{br} or~\cite[\S~17]{schrijver} for more background on the optimal assignment problem and a discussion
of alternative algorithms.

We define the \new{formal characteristic polynomial} of $A$,
\[
P_A:=\perm (\iY I \oplus A)= \bigoplus_{\sigma\in \Sym_n}
\bigotimes_{i=1}^n (\iY \delta_{i\sigma (i)} \oplus A_{i\sigma (i)})
\in \rmin[\iY] \enspace,
\]
where $I$ is the identity matrix,
and  $\delta_{ij}=\unit$ if $i=j$ and $\delta_{ij}=\zero$ otherwise.
The formal polynomial
$P_A$ has degree $n$ and its coefficients are given by
 $(P_A)_k=\trm_{n-k} (A)$, for $k=0,\ldots , n-1$ and $(P_A)_n=\unit$, where
$\trm_k(A)$ is the min-plus $k$-th trace of $A$:
\begin{equation}\label{trmin}
\trm_k(A):=
\bigoplus_{J\subset\{1,\ldots , n\},\, \card{J}= k}\left(
\bigoplus_{\sigma\in \Sym_J} \bigotimes_{j\in J}  A_{j\sigma (j)} \right)
\enspace ,
\end{equation}
where $\Sym_J$ is the set of permutations of $J$.
The associated min-plus polynomial function will be called
the \new{characteristic polynomial function} of $A$, and 
its roots will be called the \new{(algebraic) eigenvalues} of $A$.

The algebraic eigenvalues of $A$ (and so, its
characteristic polynomial function) can be computed
in $O(n^4)$ time by the method of Burkard and Butkovi\v{c}~\cite{bb02}.
Gassner and Klinz~\cite{gassner} showed
that this can be reduced to a $O(n^3)$ time, using parametric optimal 
assignment techniques. 
However, it is not known whether the sequence of coefficients
of the {\em formal} characteristic polynomial $P_A$
can be computed in polynomial time.

The term {\em algebraic eigenvalue} is used here since unlike
for matrices with real or complex coefficients, 
a root $\lambda\in\rmin$
of the characteristic polynomial of a $n\times n$ min-plus matrix
$A$ may not satisfy $Au=\lambda u$ for some 
$u\in \rmin^n$. To avoid any confusion, we shall
call a scalar $\lambda$ with the latter property
a \new{geometric eigenvalue}. 
The following statement and remarks collect
 some results in tropical spectral theory,
which have been developed by several authors~\cite{cuning,vorobyev67,romanovski,gondran77,cohen83,maslov92,bapat95,agw04,MR2587784}.
We refer the reader to~\cite{bcoq,butkovicbook} for more
information.
We say that a matrix $A$ is {\em irreducible} if $G(A)$
is strongly connected. 
\begin{theorem}[See e.g.~\cite{butkovicbook}]
The minimal algebraic eigenvalue of a matrix $A\in \rmin^{n\times n}$
is given by
\begin{equation}
\rhomin(A) =\bigoplus_{k=1}^{n} \bigoplus_{i_1,\ldots,i_k}
(A_{i_1i_2}\cdots A_{i_ki_1})^{\frac{1}{k}}
\label{e-1} \enspace ,
\end{equation}
or equivalently, by the following expression called
{\em minimal circuit mean},
\begin{equation}
\min_{c\mrm{ circuit in } G(A)} \frac{|c|_A}{|c|}
\enspace ,
\label{e-2-2} 
\end{equation}
where for all paths $p=(i_0,i_1,\ldots , i_k)$
in $G(A)$, we denote by $|p|_A=
A_{i_0i_1 }+ \cdots + A_{i_{k-1}i_k}$ the \new{weight}
of $p$, and by $|p|=k$ its \new{length}, and the minimum
is taken over all elementary circuits of $G(A)$.
\hfill\qed
\end{theorem}
An important notion to be used in the sequel
is the one of \new{critical} circuit,
i.e., of circuit $c=(i_1,i_2,\ldots, i_k,i_1)$ of
$G(A)$ attaining the minimum in~\eqref{e-2-2}.
The \new{critical graph} of $A$ is
the union of the critical circuits, that is the graph whose nodes and arcs
belong to critical circuits. 
It is known that $\rhomin(A)$ is the minimal geometric eigenvalue
of $A$ and that the multiplicity
ot $\rhomin(A)$ as a geometric eigenvalue (i.e., the ``dimension''
of the associated eigenspace) concides with the number
of strongly connected components of the critical graph.
Note also that, if $A$ is irreducible, $\rhomin(A)$
is the unique  geometric eigenvalue of $A$.

\begin{remark}
The multiplicity of $\rhomin(A)$, as an algebraic eigenvalue,
coincides with the term rank
(i.e., the maximal number of nodes of a disjoint union of circuits)
of the critical graph of $A$. This follows from the arguments
of proof of Theorem~4.7 in~\cite{abg04b}. 
\end{remark}

\section{Majorization inequalities for valuations of eigenvalues}
\label{sec-min-plus-charac}
The inequalities that we shall establish involve the notion
of {\em weak majorization}, see~\cite{MAR}
for background.
\begin{definition}
Let $u,v\in \rmin^n$. Let $u_{(1)}\leq \cdots \leq u_{(n)}$
(resp.\ $v_{(1)}\leq \cdots \leq v_{(n)}$) denote the components
of $u$ (resp.\ $v$) in increasing order.
We say that $u$ is \new{weakly (super) majorized} by $v$,
and we write $u\weakm v$, if the following conditions hold:
\[  u_{(1)}\cdots  u_{(k)}
\geq  v_{(1)} \cdots   v_{(k)}  \quad \forall k=1, \ldots, n \enspace . \]
\end{definition}
The weak majorization relation is only defined
in~\cite{MAR} for vectors of $\R^n$. Here, it is convenient
to define this notion for vectors with infinite entries.
We used the min-plus notation for homogeneity with the rest 
of the paper.
The following lemma states a useful
monotonicity property of the map which associates to a
formal min-plus polynomial $P$ its sequence
of roots, $\corn{P}$. 
\begin{lemma}\label{minpoly-maj}
Let $P, Q\in \rmin[X]$ be two formal polynomials of degree $n$.
Then,
\begin{equation}
P\geq Q \mrm{ and } P_n=Q_n \implies \corn{P}\weakm \corn{Q}
\enspace .
\end{equation}
\end{lemma}
\begin{proof}
{From} $P\geq Q$, we deduce $\divex{P}\geq \divex{Q}$.
Let $\corn{P}= (c_1(P)\leq\cdots \leq c_n(P))$
and 
$\corn{Q}=(c_1(Q)\leq\cdots\leq c_n(Q))$ denote the sequence of roots
of $P$ and $Q$, respectively.
Using $\divex{P}\geq \divex{Q}$, $\divex{P}_n=P_n=Q_n=\divex{Q}_n$ and 
\eqref{corners0}, we get
$c_1(P)\cdots c_k(P)=\divex{P}_{n-k}(\divex{P}_{n})^{-1}
\geq \divex{Q}_{n-k}(\divex{Q}_{n})^{-1}=c_1(Q)\cdots c_k(Q)$,
for all $k=1,\ldots , n$, that is $\corn{P}\prec^{\rm w} \corn{Q}$.
\end{proof}
Let $\nu$ be a (non-archimedean) valuation on a field $\K$,
i.e., a map $\nu: \K \to \mathbb{R}\cup\{+\infty\}$
satisfying the conditions~\eqref{e-intro} recalled in the introduction.
We shall think of the images of $\nu$ as elements
of the min-plus semifield, writing $\nu(ab)=\nu(a)\nu(b)$.
The main example of valuation considered here
is obtained by considering
the field of complex Puiseux series, with the valuation which
takes the smallest exponent of a series. 
Recall that this field consists of the series of the form
$ \sum_{k=K}^{\infty}  a_k \epsilon^{k/s}$ with $a_k \in \C$,
$K\in \Z$ and $s\in \N\setminus\{0\}$,
in which case the smallest exponent is equal to $K/s$
as soon as $a_K\neq 0$. 
The results of the present paper apply as well to
formal series or to series that are absolutely convergent for a sufficiently
small positive $\epsilon$.

The following proposition
formulates in terms of tropical roots a well known property usually stated
in terms of Newton polygons, see for instance~\cite[Exer.\ VI.4.11]{bourbakicomm}. It is a special case of a result proved in~\cite{kapranov}
for non-archimedean amoebas of hypersurfaces. 
We include a proof relying on Lemma~\ref{carac-ci}
for the convenience of the reader, since
we shall use the same argument in the sequel.
\begin{proposition}[See~{\cite[Th.~2.1.1]{kapranov}}]\label{prop-kap}
Let $\K$ be an algebraically closed field with a (non-archimedean) valuation $\nu$
and let $\sP=\sum_{k=0}^n\sP_k \iY^k\in \K[\iY]$, with $\sP_n=1$. Then, the 
valuations of the roots of $\sP$ (counted with multiplicities)
coincide with the roots of the min-plus polynomial $\nu(\sP):=\bigoplus_{k=0}^n\nu(\sP_k)\iY^k$.
\end{proposition}
\begin{proof}
Let $\sY_1,\ldots,\sY_n$
denote the roots of $\sP$, ordered by nondecreasing valuation,
$c_i:=\nu(\sY_i)$, so that $c_1\leq \cdots \leq c_n$,
$Q:=\bigoplus_{k=0}^nc_1\dots c_k\iY^{n-k}$, and $P:=\nu(\sP)$. Observe that 
$P_n=Q_n=\unit$, and $Q=(\iY\oplus c_1)\dots (\iY\oplus c_n)$.
Since $\sP_{n-k}=(-1)^k\sum_{i_1<\cdots<i_k}\sY_{i_1}\cdots \sY_{i_k}$,
we get $\nu(\sP_{n-k})\geq \nu(\sY_1\cdots \sY_k)=c_1\cdots c_k$, and so $P\geq Q$, and $P_n=Q_n$. Moreover, 
if $c_k<c_{k+1}$ or $k=n$, $\sY_{1}\cdots \sY_{k}$ is the only
term in the sum $(-1)^k\sum_{i_1<\cdots<i_k}\sY_{i_1}\cdots \sY_{i_k}$,
having a minimal valuation, and so, $P_{n-k}=\nu(\sP_{n-k})=c_1\cdots c_k=P_nc_1\cdots c_k$.
Then, it follows from Lemma~\ref{carac-ci} that %
$R(P)= (c_1\leq \cdots \leq c_n)=(\nu(\sY_1),\cdots,\nu(\sY_n))$.
\end{proof}

We now establish majorization inequalities for the valuations
of the eigenvalues of matrices.

\begin{theorem}\label{th-major-eig}
Let $\K$ be an algebraically closed field with a (non-archimedean) valuation $\nu$. Let  $\sA=(\sA_{ij})\in \K^{n\times n}$. Then,
the sequence of valuations of the eigenvalues of $\sA$ (counted with multiplicities) is weakly majorized by the sequence of (algebraic) eigenvalues of the matrix
$A:=(\nu(\sA_{ij}))\in \rmin^{n\times n}$.
\end{theorem}
\begin{proof}
Let $\sQ:=\det(\iY I-\sA)\in\K[\iY]$ be the characteristic polynomial of $\sA$, and let $P:=\perm (\iY I\oplus A)\in  \rmin[\iY]$ be the min-plus characteristic polynomial of $A$. Let $Q:=\nu(\sQ)$.  Observe that the coefficients of $\sQ$ are given by 
 $\sQ_k=(-1)^{n-k} \tr_{n-k} (\sA)$, for $k=0,\ldots , n-1$
and $\sQ_n=1$, where $\tr_{k}(\sA)$ is the $k$-th trace of $\sA$:
\begin{equation}\label{trusual}
\tr_{k}(\sA):=\sum_{J\subset\{1,\ldots , n\},\, \card{J}= k}\left(
\sum_{\sigma\in \Sym_J} \sgn (\sigma) \prod_{j\in J}  \sA_{j\sigma (j)} 
\right) \enspace .\end{equation}
Similarly, the coefficients of $P$ are given by
 $P_k=\trm_{n-k} (A)$, for $k=0,\ldots ,
 n-1$ and $P_n=\unit$, where $\trm_{k} (A)$ is the min-plus $k$-th trace of 
$A$~\eqref{trmin}.
It follows that $Q=\nu(\sQ)\geq P$, and $Q_n=P_n=0$. Hence,
by Lemma~\ref{minpoly-maj},
$R(Q)\prec^{\rm w} R(P)$. By Proposition~\ref{prop-kap},
$R(Q)$ coincides with the sequence of valuations of the eigenvalues
of $\sA$, which establishes the result.
\end{proof}
If the minimum in every expression~\eqref{trmin} is attained
by only one product, we have $\nu(\sQ)=P$ in the previous
proof, and so, the majorization inequality becomes
an equality. However, this condition is quite restrictive
(it requires each of a family of a combinatorial optimization
problem to have a unique solution). We shall see in the next section
that the same conclusion holds under milder assumptions.

\section{Large Deviation Type Asymptotics and Quasivaluations}
\label{sec-quasival}

The results of the previous section apply to the field
of complex Puiseux series. 
However, in some
problems of asymptotic analysis, we need to deal with complex
functions $f$ of a small positive parameter $\epsilon$ which may not have Puiseux series expansions, but which only have a ``large deviation'' type
asymptotics, meaning that the limit 
\begin{equation}\label{puis-1new}
\lim_{\eps\to 0} \frac{\log | f(\eps)|}{\log \eps}
\in \R\cup\{+\infty\}
\end{equation}
exists. Since the set of such functions is not a ring (it is 
not stable by sum), we introduce the 
larger set
$\cont$  of continuous functions $f$ defined on some interval
$(0,\eps_0)$ to $\C$ with $\eps_0>0$,
such that $|f(\eps)|\leq \eps^{-k}$ on $(0,\eps_0)$,
for some positive constant $k$.
Since all the properties that we will prove in the sequel will hold
on some neighborhoods of $0$, we shall rather use the ring of 
\new{germs} at $0$ of elements of $\cont$, which is obtained by quotienting
$\cont$ by the equivalence relation that identifies functions which coincide
on a neighborhood of $0$.
This ring of germs will be also denoted by $\cont$.
For any germ $f\in\cont$, we shall abusively denote by $f(\eps)$ or $f_\eps$
the value at $\eps$ of any representative of the germ $f$.
We shall make a similar abuse for vectors, matrices, polynomials whose
coefficients are germs. We call \new{exponent} of $f\in\cont$:
\begin{equation}\label{puis-1}
\dex{f}\bydef \liminf_{\eps\to 0} \frac{\log | f(\eps)|}{\log \eps}
\in \R\cup\{+\infty\}\enspace ,
\end{equation}
and denote by  $\cont\proper$ the subset of elements $f$ of $\cont$
having a large deviation type asymptotics, 
that is such that the liminf in the definition
of $\dex{f}$ is a limit. A convenient setting in tropical geometry,
along the lines of Alessandrini~\cite{alessandrini2013}, is to work
with functions that are definable in a o-minimal model with a polynomial growth.
Then, standard model theory arguments show that such functions have automatically large deviations type asymptotics, so that the present results apply in particular to this setting.

We have, for all $f,g\in \cont$,
\begin{subequations}
\label{defexp}
\begin{align}
\dex{f+g}\geq & \min( \dex{f}, \dex{g})\enspace,\label{puis1}\\
\dex{fg}\geq & \dex{f}+ \dex{g}\enspace,\label{puis2}
\end{align}
\end{subequations}
with 
\begin{gather}\label{equalitymin}
\dex{f+g}=\min( \dex{f}, \dex{g})\quad\text{if } \dex{f}\neq \dex{g}\enspace,
\end{gather}
and equality in~\eqref{puis2} if $f$ or $g$ belongs to $\cont\proper$.
An element $f\in \cont$ is invertible  if, and only if, there exists
a positive constant $k$ such that $|f(\eps)|\geq \eps^k$.
Then, $\dex{f}\neq \zero$, and the inverse of $f$ is the map 
$f^{-1}:\eps\mapsto f(\eps)^{-1}$. Moreover, we have $\dex{f^{-1}}\leq -\dex{f}$
with equality if, and only if,  $f\in\cont\proper$.
Thus, $f\mapsto \dex{f}$ is ``almost'' a valuation on the ring
$\cont$ (and thus almost a morphism $\cont\to\rmin$).

In the sequel, we shall say that a map $\dexo$ from 
a ring $\Ring$ to $\R\cup\{+\infty\}$ is a \new{quasi-valuation} if
it satisfies \eqref{defexp}, for all $f,g\in \Ring$, together with
\begin{equation}\label{equal0}
\dex{-1}=0\enspace .
\end{equation}
For any quasi-valuation, we define the set:
\begin{equation} \label{defacirc}
\Ring\proper:=\set{f\in \Ring}{ \dex{fg}=\dex{f} +\dex{g}\; \forall g\in\Ring}
\enspace .
\end{equation}
From the above remarks, the map $\dexo$ of~\eqref{puis-1}
is a quasi-valuation over the ring $\Ring=\cont$,
and one can easily show in that case 
that the subset $\Ring\proper$ coincides with the
set $\cont\proper$ defined above.
Another example of a quasi-valuation is the map 
\begin{equation}\label{puis-1}
\dex{f}\bydef \liminf_{\eps\to 0} \frac{\log \| f(\eps)\|}{\log \eps}
\in \R\cup\{+\infty\}\enspace ,
\end{equation}
on the ring $\Ring=\cont^{n\times n}$ of $n\times n$ matrices with entries in
$\cont$, where $\|\cdot\|$ is any matrix norm on $\C^{n\times n}$.
Indeed, there exists a constant $C$ such that $\|A B\|\leq C \|A \| \|B\|$,
for all $A,B\in\C^{n\times n}$. This property together with 
the sup-additivity of a norm imply~\eqref{defexp}.
The identity matrix is the unit of $\Ring$ and
since any constant matrix $A\in\C^{n\times n}$ satisfies $\dex{A}=0$, we
get~\eqref{equal0}.
However, the ring $\cont^{n\times n}$ is not commutative, so that the 
results of the end of the present section cannot be applied directly.

Most of the properties of the map $\dexo$ of~\eqref{puis-1}
can be transposed to the case of a general quasi-valuation, as follows.
Let us denote by  $\Ring^*$ the set of invertible elements of $\Ring$.
It is easy to see that if a map $\dexo$ from $\Ring$
to $\R\cup\{+\infty\}$ is not identically $+\infty$ and 
satisfies~\eqref{puis2}, then  $\dex{1}\leq 0$, and thus
 $f\in \Ring^* \Rightarrow\dex{f}\neq +\infty$.
The condition~\eqref{equal0} is equivalent to the
condition that $-1\in \Ring\proper$ and it
implies that $\dex{1}=0$,  $1\in \Ring\proper$, and 
$\dex{-g}=\dex{g}$ for all $g\in\Ring$.
Then, from the latter property, 
a quasi-valuation $\dexo$ satisfies necessarily~\eqref{equalitymin}.
Moreover, the set $\Ring\proper$ is necessarily a multiplicative 
submonoid of $\Ring$, 
the set $\Ring\proper\cap \Ring^*$  is the subgroup of $\Ring^*$ 
composed of the invertible elements $f$ of $\Ring$ such that
 $\dex{f^{-1}}=-\dex{f}$, 
and $\dexo$ is a multiplicative group morphism on it.

For any formal polynomial with coefficients
in a ring $\Ring$ with a quasi-valuation $\dexo$,
$\sP = \sum_{j=0}^{n} \sP_j \iY^j\in \Ring [\iY]$,
we define its quasivaluation similarly to its valuation,
see Proposition~\ref{prop-kap}:
\[ \dex{\sP}\bydef \bigoplus_{j=0}^n \dex{\sP_j} \iY^j\in \rmin[\iY]
\enspace .\]
Using the same proof, while replacing the valuation
by a quasi-valuation, we extend Proposition~\ref{prop-kap} 
and Theorem~\ref{th-major-eig} as follows.
These results hold in particular for the exponent application $\dexo$
defined on $\cont$, in which case Proposition~\ref{prop-kap-pre}
says that ``the leading exponents of the roots of a polynomial are
the min-plus roots of the polynomial of leading exponents''. 

\begin{proposition}\label{prop-kap-pre}
Let $\Ring$ be a commutative ring with a quasi-valuation $\dexo$, and let 
$\Ring\proper$ be defined by~\eqref{defacirc}.
Let $\sP=\sum_{k=0}^n\sP_k \iY^k\in \Ring[\iY]$, with $\sP_n=1$. 
Assume that $\sP$ has $n$ roots (counted with multiplicities) and that 
they all belong to $\Ring\proper$. Then, the 
images by $\dexo$ of the roots of $\sP$ (counted with multiplicities)
coincide with the roots of the min-plus polynomial $\dex{\sP}$.
\qed
\end{proposition}

\begin{theorem}\label{th-major-eig-pre}
Let $\Ring$ be a commutative ring with a quasi-valuation $\dexo$, and let 
$\Ring\proper$ be defined by~\eqref{defacirc}.
Let  $\sA=(\sA_{ij})\in \Ring^{n\times n}$.
Assume that $\sA$ has $n$ algebraic eigenvalues (counted with multiplicities)
and that  they all belong to $\Ring\proper$.  
Denote by $\Lambda=(\Lambda_1\leq \cdots \leq \Lambda_n)$ 
the sequence of their images by $\dexo$ (counted with multiplicities). Let 
$\Gamma=(\gamma_1\leq\cdots \leq \gamma_n)$ be the sequence of 
min-plus algebraic eigenvalues of $\dex{\sA}:=(\dex{\sA_{ij}})\in \rmin^{n\times n}$.
Then, $\Lambda$ is weakly majorized by $\Gamma$.
\qed
\end{theorem}

The following corollary will be used in Section~\ref{sec-gen}.
\begin{corollary}\label{cor-major-eig-pre}
Let $\Ring$ be a commutative ring with a quasi-valuation $\dexo$, and let 
$\Ring\proper$ be defined by~\eqref{defacirc}.
Let  $\sA=(\sA_{ij})\in \Ring^{n\times n}$.
Assume that $\sA$ has $n$ algebraic eigenvalues (counted with multiplicities)
and that  they all belong to $\Ring\proper$.  
Denote by $\Lambda=(\Lambda_1\leq \cdots \leq \Lambda_n)$ 
the sequence of their images by $\dexo$ (counted with multiplicities). Let 
$A\in \rmin^{n\times n}$ be such that $\dex{\sA_{ij}}\geq A_{ij}$,
for all $i,j\in [n]$, and let 
$\Gamma=(\gamma_1\leq\cdots \leq \gamma_n)$ be the sequence of 
min-plus algebraic eigenvalues of $A$.
Then, $\Lambda$ is weakly majorized by $\Gamma$.
\end{corollary}
\begin{proof}
Theorem~\ref{th-major-eig-pre} implies 
that $\Lambda \weakm\Gamma'$, where $\Gamma'$
is the sequence of min-plus algebraic eigenvalues of 
$\dex{\sA}$.
Define the min-plus polynomials $P:=\perm \dex{\sA}$ and
$Q:=\perm A$. Then, by definition of min-plus eigenvalues,
$\Gamma'$ is the sequence $\corn{P}$ of min-plus roots of $P$,
and $\Gamma$ is the sequence $\corn{Q}$ of min-plus roots of $Q$.
Since $\dex{\sA}_{ij}=\dex{\sA_{ij}}\geq A_{ij}$,
for all $i,j\in [n]$, we get that $P\geq Q$, and since $P_n=Q_n$,
Lemma~\ref{minpoly-maj} shows that $\corn{P}\weakm \corn{Q}$. Hence
$\lambda \weakm\Gamma'\weakm\Gamma$, which finishes the proof.
\end{proof}

\section{A Preliminary: Newton-Puiseux Theorem with Partial Information on Valuations}
\label{sec-newton}

If $\mathcal{A}$ is a matrix with entries in the field of Puiseux series,
the knowledge of the valuations of the entries of $\mathcal{A}$ is not
enough to determine the valuations of the coefficients
of the characteristic polynomial of $\mathcal{A}$, owing to potential cancellation. 
Therefore, in order to find conditions under which the majorization inequality in Theorem~\ref{th-major-eig-pre}  becomes an equality, we need to state a variant of the classical Newton Puiseux theorem, in which only a partial information on the valuations
of the coefficients of a polynomial, or equivalently, 
an ``external approximation'' of the Newton polytope, 
is available. The idea that such an information on the polytope is enough
to infer a partial information on roots is classical:
in the case of archimedean valuations, 
it already appeared for instance in the work of Montel~\cite{montel}.

We shall say that $f\in \cont$ has a 
\new{first order asymptotics} 
if
\begin{align}\label{puis4}
f(\eps)\sim a \eps^A, \qquad \mrm{when } \eps\to 0^+\enspace ,
\end{align}
with either $A\in \R$ and $a\in \C\setminus\{0 \}$,
or $A=+\infty$ and $a\in \C$. 
In the first case,~\eqref{puis4} means that
$\lim_{\eps\to 0} \eps^{-A}f(\eps) =a$, in the second case,~\eqref{puis4} 
means that $f=0$ (in a neighborhood of $0$).
Such asymptotic behaviors arise when considering
precise large deviations. We have:
\begin{align}
f(\eps)\sim a\eps^A & \implies  \dex{f}=A \text{ and } f\in \cont\proper\enspace.\label{puis0}
\end{align}
We shall also need a relation slightly weaker 
than $\sim$. If $f\in \cont$, $a\in \C$ and $A\in \rmin$, we write
\begin{equation}
f(\eps) \simeq a\eps^A
\label{e-f}
\end{equation}
if $f(\eps) = a\eps^A+o(\epsilon^A)$.
If $A\in\R$, this means that $\lim_{\eps\to 0} \eps^{-A}f(\eps) =a$.
If $A=+\infty$, this means by convention that $f=0$.
If $a\neq 0$ or $A=+\infty$, 
then $f(\eps) \simeq a\eps^A$ if, and only if, $f(\eps) \sim a\eps^A$
and in that case $\dex f =A$.
In general,
\begin{align}
f(\eps)\simeq a\eps^A & \implies  \dex{f}\geq A \enspace.\label{puis0-1}
\end{align}
Conversely, $\dex{f}>A\implies f(\eps)\simeq 0 \eps^A$. Of course, 
in~\eqref{e-f}, $a\eps^A$ must be viewed as a formal
expression, for the relation to be meaningful when $a=0$ and $A\in\R$.
In~\eqref{puis0}, however, $a\eps^A$ can  be viewed either as a formal
expression or as an element of $\cont$.

The following results give conditions under which some or all roots
of a polynomial with coefficients in $\cont$ have first order asymptotics,
hence are elements of $\cont\proper$, which allow one in particular to apply
the results of Section~\ref{sec-quasival}.
Although stated for polynomials with coefficients in $\cont$,
they are already useful in the case of polynomials with coefficients
in the set of Puiseux series, for which some of the valuations are not known.
lts 
We shall use in particular these results in the case of
characteristic polynomials.

\begin{theorem}[Newton-Puiseux theorem with partial information on the valuations] 
\label{th3-local}
Let $\sP= \sum_{j=0}^{n} \sP_j\iY^j\in \cont [\iY]$.
Assume that there exist $p=\sum_{j=0}^{n} p_j\iY^j\in \C [\iY]$
and $P=\bigoplus_{j=0}^{n} P_j\iY^j\in \rmin [\iY]$ satisfying
$\sP_j(\eps)\simeq p_j \eps^{P_j}$, $j=0,\ldots , n$.
Let $c\in\R$ be a finite root of $P$ with multiplicity $m$ and
assume that the polynomial
\begin{equation}\label{defpc}
p^{(c)}=\sum_{\scriptstyle 0\leq j\leq n\atop\scriptstyle \eval{P}(c)=P_j c^j} p_j\iY^j\in \C[\iY]
\enspace
\end{equation}
is not identically zero.
Let $y_1,\ldots, y_{\ell}$ denote its non-zero roots
 (counted with multiplicities). 
Then, $\ell \leq m$, and 
there exist $\ell$ roots
of $\sP$  (counted with multiplicities),
 $\sY_1,\ldots, \sY_\ell\in \cont$, having first order asymptotics of the form
$\sY_i\sim y_i \eps^c$.
Moreover if  $v=\val p^{(c)}>0$, 
then $v\geq m'$,  where $m'$  is the sum of the multiplicities of all
the roots of $P$ greater than $c$ ($+\infty$ comprised), and there exist
$v$ roots of $\sP$ (counted with multiplicities) 
$\sY_{1+\ell},\ldots, \sY_{v+\ell}\in \cont$, such that
 $\sY_i\simeq 0 \eps^c$, for $1\leq i-\ell\leq v$.
Finally, the remaining $n-v-\ell$ roots $\sY$ of  $\sP$  are such that 
$\sY \eps^{-c}$ tends to infinity when $\eps$ goes to $0$,
and their number satifies $n-v-\ell\geq n-m-m'$.
\end{theorem}

Recall that to $P$ is associated the polynomial
function $\eval{P}$ and the
convex formal polynomial $\divex{P}$, as in Section~\ref{sec-minpol}.

\begin{theorem}[Newton-Puiseux theorem with partial information on the valuations, continued]
\label{th3}
Let $\sP=\sum_{j=0}^{n} \sP_j\iY^j\in \cont [\iY]$,
such that $\sP_n=1$.
The following assertions are equivalent:
\begin{enumerate}
\item\label{th3-1} There exist $\sY_1,\ldots, \sY_n\in \cont$
such that $\sY_1(\eps),\ldots, \sY_n(\eps)$
are the roots of $\sP(\eps)=\sum_{j=0}^{n} \sP_j(\eps)\iY^j$
 counted with multiplicities, and
$\sY_1,\ldots, \sY_n$ have first order  asymptotics,
$\sY_j(\eps)\sim y_j \eps^{Y_j}$ with $Y_1\leq \cdots \leq Y_n$;
\item\label{th3-2} There exist $p=\sum_{j=0}^{n} p_j\iY^j\in \C [\iY]$
and $P=\bigoplus_{j=0}^{n} P_j\iY^j\in \rmin [\iY]$ satisfying
$\sP_j(\eps)\simeq p_j \eps^{P_j}$, $j=0,\ldots , n$,
with $p_n=1$, $P_n=\unit$, $p_0\neq 0$ or $P_0=\zero$, and 
$p_{n-i}\neq 0$ for all $i\in\{1,\ldots, n-1\}$ such that 
$c_i<c_{i+1}$, where $(c_1\leq \cdots \leq c_n)=\corn{P}$.
\end{enumerate}
When these assertions hold, we have 
$\dex{\sP}\geq P$, $\divex{\dex{\sP}}=\divex{P}$, and
$\corn{\dex{\sP}}=\corn{P}=(c_1\leq \cdots \leq c_n)=
(Y_1\leq \cdots \leq Y_n)$.
Moreover, if $c\in\R$ is a root of $P$ with multiplicity $m$ and
$c_{i+1}=\cdots =c_{i+m}=c$, then $y_{i+1},\ldots, y_{i+m}$ are precisely
the non-zero roots of the polynomial $p^{(c)}$ of~\eqref{defpc},
counted with multiplicities.
\end{theorem}

Theorem~\ref{th3} is a ``precise
large deviation'' version of the Newton-Puiseux theorem:
we assume only the existence of asymptotic equivalents
for the coefficients of $\sP$,
and derive the existence of asymptotic equivalents
for the branches of $\sP(\epsilon,y)=0$.
The Newton-Puiseux algorithm is sometimes presented
for asymptotic expansions, as in~\cite{dieu}.
The interest of the statements of Theorems~\ref{th3-local}
and~\ref{th3}, is to show that if some
coefficients are known to be negligible,
the asymptotics of the roots is determined only from the 
asymptotics of those coefficients $\mathcal{P}_i$ such that
$(i,P_i)$ is an exposed
point of the epigraph of $\divex{P}$.

\begin{example}
Consider $\sP=\iY^3 +\eps^5 \iY^2 -\eps^6 \iY + \eps^{13}$.
Then, $\sP$ is a polynomial over the field $\K$
of complex Puiseux series, hence the roots of $\sP$ are 
elements of $\K$.
Moreover, the min-plus polynomial $P=\nu(\sP)
= \iY^3 \oplus 5\iY^2 \oplus 6\iY + 13$
is the one of Figure~\ref{minpoly},
hence its roots are $c_1=c_2=3$ and $c_3=7$.
In that case, Proposition~\ref{prop-kap} says that
the valuations of roots of $\sP$ coincide with the roots of $P$.

If now  $\sP=\iY^3 +\frac{\eps^3}{\log \eps} \iY^2 - \eps^6 \iY+ \eps^{13}$,
the coefficient $\sP_2$ of $\sP$ is not in  $\K$.
Moreover, although all coefficients of $\sP$ belong to $\cont$
and even $\cont\proper$,
$\sP_2$ does not have a first order asymptotics.
However, $\sP_j(\eps)\simeq p_j \eps^{P_j}$, $j=0,\ldots , 3$,
with $P=\iY^3 \oplus 3\iY^2 \oplus 6\iY \oplus 13$
and $p=\iY^3 - \iY+ 1\in \C [\iY]$.
The min-plus polynomial $P$ has same roots as the one defined above:
$c_1=c_2=3<c_3=7$ and since $p_0\neq 0$, $p_1\neq 0$, $p_3=1$ and $P_3=\unit$,
Assertion~\eqref{th3-2} of Theorem~\ref{th3} holds.
Hence, by Theorem~\ref{th3},  the roots of $\sP$ 
form  $3$ continuous branches around $0$:
$\sY_1,\ldots, \sY_3\in \cont$ with first order  asymptotics,
and respective exponents $c_1=c_2=3<c_3=7$.

In both cases above,  Theorem~\ref{th3} gives the additional
information that the roots $\sY_1,\ldots, \sY_3$ of $\sP$ 
satisfy $\sY_j(\eps)\sim y_j \eps^{c_j},\; j=1,\ldots,3$,
where $y_1,y_2$ are the non-zero roots of $p^{(3)}=\iY^3 - \iY$,
and $y_3$ is the non-zero root of $p^{(7)}= - \iY+ 1$.
This gives for instance $y_1=1$, $y_2=-1$ and $y_3=1$,
so that $\sY_1\sim \eps^3$, $\sY^2\sim-\eps^3$
and $\sY_3\sim \eps^7$. 
\end{example}

In order to prove Theorems~\ref{th3-local} and \ref{th3},
we need the following standard result.

\begin{lemma}\label{cont-roots}
Let $\sQ(\eps,\iY)=\sum_{i=0}^n \sQ_j(\eps) \iY^j$, where the $\sQ_j$ are 
continuous functions of $\eps\in [0,\eps_0)$,
assume $\sQ(0,\cdot)\neq 0$ and let $d=\deg \sQ(0,\cdot)\geq 0$.
Then, for any open ball $B$ containing the roots of $\sQ(0,\cdot)$,
there are $d$ continuous branches $\sZ_1,\ldots ,\sZ_d$ defined in 
some interval $[0,\eps_1)$, with $0<\eps_1\leq\eps_0$, such that  
$\sZ_1(\eps),\ldots ,\sZ_d(\eps)$ are exactly the roots of 
$\sQ(\eps,\cdot)$ in $B$ counted with multiplicities.
Moreover, the roots of $\sQ(\eps,\cdot)$ that are outside $B$
tend to infinity when $\eps$ goes to $0$.
\end{lemma}
\begin{proof}
We only sketch the proof, which is classical.
By the Cauchy index theorem, 
if $\gamma$ is any circle in $\C$
containing no roots of $\sQ(\epsilon,\cdot)$,
the number of roots of $\sQ(\epsilon,\cdot)$
inside $\gamma$ 
is $(2\pi i)^{-1}\int_\gamma \partial_z \sQ(\epsilon,z)(\sQ(\epsilon,z))^{-1}\,\mrm{d}z$.
By continuity of $\eps\mapsto \sQ(\eps,\cdot)$, 
the number of roots of $\sQ(\eps',\cdot)$
inside $\gamma$ (counted with multiplicities) is constant for $\eps'$ in
some neighborhood of $\eps$. Taking $B$ as in the lemma,
$\gamma=\partial B$, and $\eps=0$, we get exactly $d$  roots of $\sQ(\eps',\cdot)$
in $B$ for $\eps'$ in some interval $[0,\eps_1)$.
Consider now a ball $B_R\supset B$ of radius $R$. 
For $\eps'$ small enough, the number of roots of $\sQ(\eps',\cdot)$ in
either $B_R$ or $B$ is equal to $d$, hence
any root of $\sQ(\eps',\cdot)$ outside $B$ must be outside $B_R$.
This shows that the roots of $\sQ(\eps',\cdot)$ 
that do not belong to $B$ go to infinity, when $\eps'\to 0$.
Finally, by taking small balls around each root of $\sQ(\eps,\cdot)$,
with $0\leq \eps< \eps_1$, we see that the map which sends $\eps$
to the unordered $d$-tuple of roots of $\sQ(\eps,\cdot)$ that belong to $B$,
is continuous on $[0,\eps_1)$.
By a selection theorem for unordered $d$-tuples depending continuously on a
real parameter (see for instance~\cite[Ch. II, Section~ 5, 2]{kato}),
we derive the existence of the $d$ continuous branches $\sZ_1,\ldots ,\sZ_d$.
\end{proof}

\begin{proof}[Proof of Theorem~\ref{th3-local}]
This is obtained by applying the
first step of the Puiseux algorithm (which is part of the proof
of the classical Newton-Puiseux theorem).
Indeed, applying the change of variable $y=z \eps^{c}$, 
and the division of $\sP$ 
by $\eps^{\eval{P}(c)}$, transforms the equation $\sP(\eps,y)
:=\sum_{j=0}^{n} \sP_j(\eps)y^j=0$ into
an equation $\sQ(\eps,z)=0$, where $\sQ(\cdot,z)$ extends 
continuously to $0$ with $\sQ(0,z)=p^{(c)}(z)$.
Since $p^{(c)}$ is not identically zero,
Lemma~\ref{cont-roots} implies that for $d=\deg p^{(c)}$, 
 there exist $\sZ_1,\ldots ,\sZ_d\in \cont$
 such that  for all $\eps\geq 0$ small enough,
$\sZ_1(\eps),\ldots ,\sZ_d(\eps)$ are exactly the roots
(counted with multiplicities) of
$\sQ(\eps,\cdot)$ in some ball around $0$,
and that the roots of $\sQ(\eps,\cdot)$ outside this ball
tend to infinity when $\eps$ goes to $0$.

In particular $\sZ_1(0),\ldots ,\sZ_d(0)$ are the roots
of $p^{(c)}$, then if $v=\val p^{(c)}$, 
 $\ell=d-v$ is the number of non-zero roots
and one can assume that $\sZ_i(0)=y_i$ for $i\leq \ell$.
Making the reverse change of variable, we obtain that
$\sY_i(\eps)=\sZ_i(\eps)\eps^{c}$, $i=1,\ldots d$
 satisfy the conditions of the theorem.

It remains to show that $\ell\leq m$,
$v\geq m'$,  and $n-v-\ell\geq n-m-m'$,
  where $m'$  is the sum of the multiplicities of all
the roots of $P$ greater than $c$.
By Corollary~\ref{carac-mul},
we have $\eval{P}(c)<P_j c^j$ when $j<m'$ or $j>m+m'$,
from which we deduce that  $v=\val p^{(c)}\geq m'$ and
$d=\deg p^{(c)}\leq m'+m$, so that $\ell=d-v\leq m$, 
and $n-v-\ell=n-d\geq n-m-m'$,
which finishes the proof.
\end{proof}

\begin{proof}[Proof of Theorem~\ref{th3}]
We first prove~\eqref{th3-1}$\implies$\eqref{th3-2}.
Let $Q=(\iY\oplus Y_1)\cdots (\iY\oplus Y_n)$. Then, $Q=\divex{Q}$,
$\corn{Q}=(Y_1\leq \cdots \leq Y_n)$ and
$Q_{n-i}= Y_1\cdots Y_i$ for all $i=1,\ldots , n$.
Since $\sY_1(\eps),\ldots, \sY_n(\eps)$
are the roots of $\sP(\eps, y)=0$ counted with multiplicities, 
and $\sP_n=1$, it follows that
$\sP(\eps, \iY) =\prod_{i=1}^n (\iY- \sY_i(\eps)) $.
Hence, $(-1)^i\sP_{n-i}$ is the sum of all
products $\sY_{j_1}\cdots \sY_{j_i}$, where $j_1,\ldots, j_i$
are pairwise distinct elements of $\{1,\ldots , n\}$.
By the properties of ``$\simeq$'' (stability by addition and
multiplication), and since $\bigoplus_{j_1,\ldots , j_i}
 Y_{j_1}\cdots Y_{j_i}=Y_1\cdots Y_i
=Q_{n-i}$, we obtain that there exist $p_0,\ldots, p_{n-1}\in \C$
such that $\sP_{j}\simeq p_j \eps^{Q_j}$ for all $j=0,\ldots, n-1$.
Putting $p_n=1$, we also get $\sP_{n}=1\simeq p_n \eps^{Q_n}$ 
since $Q_n=\unit$.
When $i=1,\ldots , n-1$ is such that
$Y_i<Y_{i+1}$, $\sY_{1}\cdots \sY_{i}$ is the only leading term 
in the sum of all $\sY_{j_1}\cdots \sY_{j_i}$,
and then $p_{n-i}=(-1)^i y_1\cdots y_i\neq 0$.
Moreover, for $i=n$, either $Y_n\neq\zero$, which implies that
$p_{0}=(-1)^n y_1\cdots y_n\neq 0$, or $Y_n=\zero$, which implies
that $\sY_n=0$, $\sP_0=0$ and $Q_0=\zero$.
This shows that $(c_1,\ldots,c_n)=(Y_1,\ldots,Y_n)$
and $P=Q$ are as in Point~\eqref{th3-2}.

The implication~\eqref{th3-2}$\implies$\eqref{th3-1}
and the remaining part of the theorem will follow from several applications of 
Theorem~\ref{th3-local}.
We now only assume that the $\sP_j\in\cont$ 
satisfy Point~\eqref{th3-2} and consider
$(c_1\leq \cdots \leq c_n)=\corn{P}$.
From Lemma~\ref{carac-ci} applied to $P$, we get that
$P\geq P_n (\iY\oplus c_1)\cdots (\iY\oplus c_n)$,
$P_{n-i}=P_n c_1\cdots c_i$
for all $i\in\{0,n\}\cup\set{i\in \{1,\ldots , n-1\}}{c_i< c_{i+1}}$,
and $\divex{P}=P_n (\iY\oplus c_1)\cdots (\iY\oplus c_n)$.
It follows from~{\rm (\ref{defexp},\ref{puis0-1})}, 
that $\dex{\sP}\geq P\geq P_n (\iY\oplus c_1)\cdots (\iY\oplus c_n)$,
and from  Point~\eqref{th3-2} and~\eqref{puis0},
we get that $\dex{\sP}_{n-i}= P_{n-i}$
for all $i\in\{0,n\}\cup\set{i\in \{1,\ldots , n-1\}}{c_i< c_{i+1}}$.
Then, $\dex{\sP}$ satifies the conditions of  Lemma~\ref{carac-ci},
which yields $\corn{\dex{\sP}}=\corn{P}$ and
$\divex{\dex{\sP}}=\divex{P}$.

Let $c\neq\zero$ be  a root of $P$ with multiplicity $m$.
Then, there exists $0\leq i <n$ such that  $c=c_{i+1}=\cdots =c_{i+m}$.
Using the same arguments as in the proof of Theorem~\ref{th3-local},
we have that $i=n-m'-m$, where $m'$ is the sum of the multiplicities of all
the roots of $P$ greater than $c$, 
$\val p^{(c)}\geq m'=n-m-i$ and $d=\deg p^{(c)}\leq m'+m=n-i$.
Moreover, since  either $i=0$ or $c_{i}<c_{i+1}$,
we get, by Point~\eqref{th3-2}, that $p_{n-i}\neq 0$, hence $\deg p^{(c)}=n-i$.
Similarly, we have either $i+m=n$ or $c_{i+m}<c_{i+m+1}$.
In the second case, we get $p_{n-i-m}\neq 0$, thus $\val p^{(c)}=n-m-i$.
In the first case, $n-m-i=0$ and either $p_{0}\neq  0$ or $P_{0}=\zero$.
Since $P_{0}=\zero$ implies $c=c_n=\zero$, which contradicts our
assumption, we must have $p_0\neq 0$, hence again $\val p^{(c)}=n-m-i$.
This shows that $\deg p^{(c)}-\val p^{(c)}=m$, so that
$ p^{(c)}$ has $\ell=m$ non-zero roots and applying 
Theorem~\ref{th3-local}, we get that
if $y_{i+1},\ldots, y_{i+m}$ denote its non-zero roots,
then there exist $m$ roots
of $\sP$  (counted with multiplicities),
 $\sY_{i+1},\ldots, \sY_{i+m}\in \cont$,
 having first order asymptotics of the form
$\sY_j\sim y_j \eps^c$, $i+1\leq j\leq i+m$.

Finally, if $c=\zero$ is a root of $P$ with multiplicity $m$, then 
$c_{n-m}<c_{n-m+1}=\cdots =c_n=\zero$, $\val P=m$, and
$P$ has $n-m$ roots $\neq\zero$.
This implies that $\val \sP\geq m$, so that $0$ is
a root of $\sP$ with multiplicity $\geq m$.
Moreover, we have shown above that there exist $n-m$ roots of $\sP$ with first
order asymptotics with respective 
exponents $c_1\leq \cdots \leq c_{n-m}<+\infty$.
Hence, $\sP$ cannot have more than $m$ zero roots, so that $\sP$  has exactly
$m$ roots with first order asymptotics of the form $Y_j\sim 0$
(and $\val \sP= m$).
We thus have shown Point~\eqref{th3-1}, which finishes the proof of the theorem.
\end{proof}

\section{Genericity for asymptotics of eigenvalues}
\label{sec-gen}

We consider a matrix $\sA\in \cont^{n\times n}$ 
and we shall assume that the entries $(\sA_\eps)_{ij}$ of $\sA_\eps$
have asymptotics of the form:
\begin{align}
&(\sA_\eps)_{ij}\simeq a_{ij} \eps^{A_{ij}},\;\label{assump1} 
\mrm{for some matrices }\\& a=(a_{ij})\in \C^{n\times n},%
\mrm{ and }%
A=(A_{ij})\in \rmin^{n\times n}.\nonumber
\end{align}
Applying Corollary~\ref{cor-major-eig-pre}, we obtain the
following majorization inequality.
\begin{theorem}\label{th-major-cont}
Let  $\sA\in \cont^{n\times n}$ satisfy~\eqref{assump1}
and  let 
$\Gamma=(\gamma_1\leq\cdots \leq \gamma_n)$ be the sequence of 
min-plus algebraic eigenvalues of $A$.
Assume that the eigenvalues $\sL^1_\eps,\ldots , \sL^n_\eps$
of $\sA_\eps$ (counted with multiplicities) have first order asymptotics,
\begin{align*}%
\sL_\eps^i\sim \lambda_i \eps^{\Lambda_i}\enspace ,
\end{align*}
and let $\Lambda = (\Lambda_1 \leq \cdots \leq \Lambda_n)$.
Then, $\Lambda  \weakm\Gamma$.
\end{theorem}
\begin{proof}
From~\eqref{puis0}, the  eigenvalues of $\sA_\eps$
are elements of $\cont\proper$
and their images by $\dexo$ (counted with multiplicities)
are equal to $\Lambda_1\leq \cdots \leq \Lambda_n$.
By~\eqref{puis0-1} and~\eqref{assump1}, we have $\dex{\sA}_{ij}\geq A_{ij}$,
for all $i,j\in [n]$.
Then, applying Corollary~\ref{cor-major-eig-pre} to the ring $\cont$,
and the matrices $\sA$ and $A$, 
we deduce that $\Lambda \weakm\Gamma$.
\end{proof}

\begin{remark}\label{rem-puis}
If the coefficients of $\sA_\eps$ have Puiseux series expansions in $\eps$, 
the coefficients $\sQ_j$ of the characteristic 
polynomial $\sQ(\eps,\iY)$  of $\sA$ and thus the eigenvalues of $\sA$ 
belong to $\cont$ and have first order asymptotics,
so that Theorem~\ref{th-major-cont} applies.
If we only assume that $\sA\in\cont^{n\times n}$ 
satisfies~\eqref{assump1}, then the coefficients $\sQ_j$,
which are elements of $\cont$, 
need not have first order asymptotics (even if $a_{ij}\neq 0$ for all
$i,j$) due to cancellations. 
However, they satisfy the conditions $\sQ_n= 1$
and $\sQ_j(\eps)\simeq q_j \eps^{Q_j}$ for some exponents 
$Q_j\in\rmin$ computed using 
the exponents $A_{ij}$ (see Section~\ref{sec-min-plus-charac}),
so that Theorem~\ref{th3} can be applied, leading to some sufficient
conditions for the first order asymptotics of the eigenvalues of $\sA_\eps$
to exist. This will be used in Theorem~\ref{th-gen=}.
\end{remark}

We next show that the weak majorization inequality is 
an inequality in generic circumstances.
Formally, we shall consider the following notion of genericity.
We will say that a property $\ptyP(y)$ depending on the variable $y=(y_1,\ldots
, y_n)\in \C^n$ holds for \new{generic values} of $y$ if the set 
of elements  $y\in \C^{n}$ such that the property $\ptyP(y)$
is false is included in a proper algebraic set.
This means that there exists $Q\in\C[\iY_1,\ldots, \iY_n]\setminus\{0\}$ 
such that $\ptyP(y)$ is true if $Q(y)\neq 0$.
When the parameter $y$ will be obvious, we shall simply say that
$\ptyP$ is \new{generic} or holds \new{generically}.
It is clear that if $\ptyP_1$ and $\ptyP_2$ are both generic, then
``$\ptyP_1$ and $\ptyP_2$'' is also generic.

Since any polynomial $q=\sum_{i_1,\ldots, i_n\in\N} 
q_{i_1,\ldots, i_n} \iY_1^{i_1}\cdots \iY_n^{i_n} 
\in \C[\iY_1,\ldots, \iY_n]$ in $n$ indeterminates
can be seen as
an element of  $\cont [\iY_1,\ldots, \iY_n]$ whose coefficients are constant 
with respect to $\eps$, we have:
\begin{align}\label{qQ}
\dex{q}=\bigoplus_{\scriptstyle i_1,\ldots, i_n\in\N\atop\scriptstyle q_{i_1,\ldots, i_n}\neq 0}
\iY_1^{i_1}\cdots \iY_n^{i_n} \enspace \in \rmin[\iY_1,\ldots, \iY_n]
\enspace .\end{align}
We also define, for any $Y\in\rmin^n$, 
\begin{align}\label{qsat}
q^{\sat}_Y:=\sum_{\scriptstyle i_1,\ldots, i_n\in\N\atop\scriptstyle \eval{\dex{q}}(Y_1,\ldots , Y_n)=
Y_1^{i_1}\cdots Y_n^{i_n}\neq \zero } q_{i_1,\ldots, i_n} \iY_1^{i_1}\cdots
\iY_n^{i_n} \in \C[\iY_1,\ldots, \iY_n] \enspace .
\end{align}

The following result is clear from the above definitions of $\dex{q}$
and $q^\sat_Y$, since when $y\neq 0$ or $Y=\zero$, $\sY\simeq y \eps^Y\iff
\sY\sim y \eps^Y$.
Note that there and in the sequel, we use the same notation for any
formal polynomial over $\C$ or $\cont$ and its associated polynomial 
function.

\begin{lemma}\label{gener}
Let $q\in \C[\iY_1,\ldots, \iY_n]$ be non zero and  let $Q=\dex{q}$ and
$q^\sat_Y$ be defined by~\eqref{qQ} and~\eqref{qsat}, respectively.
Let $\sY\in \cont^n$, $y\in \C^n$ and $Y\in \rmin^n$ be 
such that $\sY_i\simeq y_i \eps^{Y_i}$ for $i=1,\ldots , n$. Then, 
\begin{align}\label{simpoly}
 q(\sY_1,\ldots, \sY_n)&\simeq q^{\sat}_Y(y) \eps^{\eval{Q}(Y)}\enspace ,
\end{align}
and for any fixed $Y$ such that $\eval{Q}(Y)\neq \zero$,
we have $q^{\sat}_Y(y)\neq 0$ for generic values of $y\in \C^n$.
Then, for generic values of $y\in \C^n$, the relation $\simeq$
in~\eqref{simpoly} can be replaced by an the asymptotic equivalence relation $\sim$. 
\end{lemma}

\begin{theorem}\label{th-gen=}
Let  $\sA\in \cont^{n\times n}$ satisfy~\eqref{assump1}
and  let 
$\Gamma=(\gamma_1\leq\cdots \leq \gamma_n)$ be the sequence of 
min-plus algebraic eigenvalues of $A$.
For generic values of $a=(a_{ij})\in \C^{n\times n}$, 
the eigenvalues $\sL^1_\eps,\ldots , \sL^n_\eps$
of $\sA_\eps$ (counted with multiplicities) have first order asymptotics,
\begin{align*}%
\sL_\eps^i\sim \lambda_i \eps^{\Lambda_i}\enspace ,
\end{align*}
and 
$\Lambda = (\Lambda_1 \leq \cdots \leq \Lambda_n)$
satisfies $\Lambda=\Gamma$.
\end{theorem}
\begin{proof}
Since $\sA=\sA_\epsilon\in \cont^{n\times n}$ and $\cont$ is a ring,
the characteristic polynomial of $\sA$,
$\sQ(\eps,\iY):=\det (\iY I- \sA_\eps)$ belongs to $\cont[\iY]$.
Let $Q=\dex{\sQ}\in \rmin[\iY]$, and denote by
$P=\perm (\iY I\oplus A)\in  \rmin[\iY]$
the min-plus characteristic polynomial of $A$.
The coefficients of $\sQ$ are given by 
 $\sQ_k(\eps)=(-1)^{k} \tr_{n-k} (\sA_\eps)$, for $k=0,\ldots , n-1$
and $\sQ_n=1$, where $\tr_k$ is the usual $k$-th trace~\eqref{trusual}.
The coefficients of $P$ are given by  $P_k=\trm_{n-k} (A)$, for $k=0,\ldots ,
 n-1$ and $P_n=\unit$, where here 
$\trm_k$ is the min-plus $k$-th trace~\eqref{trmin}.
By Lemma~\ref{gener}, we obtain that for any fixed %
matrix $A\in\rmin^{n\times n}$, and any $\sA\in\cont^{n\times n}$
satisfying~\eqref{assump1} with $a\in\C^{n\times n}$ and $A$,
$\tr_k(\sA_\eps)\sim (\tr_k)^\sat_A(a) \eps^{\trm_{k}(A)}$,
for generic values of $a\in \C^{n\times n}$.
In particular, generically, 
$\sQ_k(\eps)$ has first order asymptotics and
$\dex{\sQ_k}=P_k$, for all $k=0,\ldots , n$. This implies that
$Q=P$, thus $\corn{Q}=\corn{P}=\Gamma$.
Moreover, $\sQ$ satisfies Point~\eqref{th3-2} of Theorem~\ref{th3}.
Hence, by Theorem~\ref{th3}, 
the eigenvalues $\sL^1_\eps,\ldots , \sL^n_\eps$
of $\sA_\eps$ have first order asymptotics
and their exponents $\Lambda_1\leq \cdots \leq \Lambda_n$
are equal to the roots of $Q$, hence $\Lambda=\Gamma$.
\end{proof}

\begin{remark}
Since min-plus eigenvalues can be
computed in polynomial time, see Section~\ref{sec-tropeig},
Theorem~\ref{th-gen=} shows that
the sequence $\Lambda$ of generic exponents of the eigenvalues
can be computed in polynomial time. 
\end{remark}

\section{Eigenvalues of matrix polynomials}
\label{sec-pencils}

We consider now a matrix polynomial
\begin{equation}\label{defAepsilon}
\sA= \sA_{0}+ \iY \sA_{1} +\cdots+\iY^d\sA_{d}  \enspace ,
\end{equation}
with coefficients $\sA_{k}\in \cont^{n\times n}$, $k=0,\ldots, d$.
Making explicit the dependence in the parameter $\epsilon$,
we shall write
\[
\sA_\epsilon = \sA_{\epsilon,0}+ \iY \sA_{\epsilon,1} +\cdots+\iY^d\sA_{\epsilon,d}  \enspace ,
\]
where, for every $k=0,\ldots, d$, $\sA_{\epsilon,k}$ is
a $n\times n$ matrix whose coefficients, $(\sA_{\epsilon,k})_{ij}$,
are complex valued continuous functions
of the nonnegative parameter $\epsilon$.
We shall assume that for every $0\leq k\leq d$, matrices
$a_k=((a_k)_{ij})\in \C^{n\times n}$ and $A_k=((A_k)_{ij})\in \rmin^{n\times n}$ are given, so that
\begin{equation}\label{pencil-asymp}
(\sA_{\eps,k})_{ij} \simeq (a_k)_{ij}\epsilon^{(A_k)_{ij}}\enspace,
\qquad \mrm{ for all } 
 1\leq i,j\leq n \enspace .
\end{equation}

We shall also assume that the matrix polynomial $\sA_\epsilon$
is {\em regular}, which means that the characteristic polynomial  $\det(\sA_\epsilon)$
(or $\det(\sA)$) of the matrix polynomial $\sA_\epsilon$
is non identically zero.
Then, the \new{eigenvalues} of $\sA_\epsilon$ (or $\sA$) are by definition the
roots of this polynomial. If $\deg\det(\sA)<nd$, then
we also say that $\infty$ is an eigenvalue of $\sA$ with 
multiplicity $nd-\deg\det(\sA)$.
When $\sA=\sA_{0}- \iY\id$, these are the usual eigenvalues
of $\sA_{0,\epsilon}$.
We shall study here the first order asymptotics of
the eigenvalues of $\sA_\epsilon$ 
in the same spirit as in the previous section.

To the matrix polynomial $\sA$, one can associate the
{\em min-plus matrix polynomial}
$\dex{\sA}:= \dex{\sA_{0}}\oplus \iY\dex{ \sA_{1}}\oplus\cdots\oplus\iY^d\dex{\sA_{d} }$.
Here, we shall rather consider the min-plus matrix polynomial
associated to the asymptotics~\eqref{pencil-asymp}:
\begin{equation}\label{defAtrop}
A=A_0\oplus \iY A_1 \oplus \cdots  \oplus \iY^dA_d\in \rmin^{n\times n}[\iY]
\enspace.\end{equation}
The min-plus matrix polynomial $A$ 
can be seen either as a (formal) polynomial with coefficients 
in $\rmin^{n\times n}$, namely the $A_k$, $k=0,\ldots, d$,
or as a matrix with entries in $\rmin[\iY]$, denoted $A_{ij}$.
We denote by $\eval{A}$ the function
 which associates to $y\in\rmin$, the matrix
$\eval{A}(y):= A_0\oplus y A_1 \oplus \cdots  \oplus y^dA_d\in \rmin^{n\times n}$.
We call \new{min-plus characteristic polynomial} of the matrix polynomial $A$,
the permanent \(
P_A = \perm A\). This is a formal polynomial, the  associated
polynomial function of which is equal to
\( \eval{P_A}(y) = \perm \eval{A}(y)\).

We shall say that the matrix polynomial $A$ is {\em regular} if $P_A$ is
not identically zero. 
Then, the min-plus roots of $P_A$ will be called the
{\em algebraic eigenvalues} of the min-plus matrix polynomial $A$. 
Note that the valuation $\val P_A$ can be computed
by introducing the matrix $\val A\in \rmin^{n\times n}$,
such that $(\val A)_{ij}=\val A_{ij}$,
where $A_{ij}$ denotes the $(i,j)$ entry of the min-plus polynomial $A$.
Then, $\val P_A$ is equal to the min-plus permanent of the matrix $\val A$.
By symmetry, the degree $\deg P_A$ is equal to the 
max-plus permanent of the matrix
$\deg A \in \rmax^{n\times n}$,
such that $(\deg A)_{ij}=\deg A_{ij}$.
When  $\val P_A>0$, $\zero=+\infty$ is an eigenvalue of $A$
with multiplicity $\val P_A$. When $\deg P_A<nd $, 
$P_A$ has only $\deg P_A$ eigenvalues (belonging to $\rmin$).
One can say that $-\infty$ (the
infinity of $\rmin$) is an eigenvalue of $A$ with multiplicity 
$nd- \deg P_A$.
Recall that, for any scalar $y\in\rmin$,
we have  \( \eval{P_A}(y) = \perm \eval{A}(y)\)
which is the value of the optimal assignment associated to the
matrix $\eval{A}(y)$. 
Hence, 
the algebraic eigenvalues of the
matrix polynomial $A$ (and so, the polynomial
function $\eval P_A$) can be computed
in $O(n^4d)$ time by adapting
the method of Burkard and Butkovi\v{c}~\cite{bb02}.
Moreover, by adapting the parametric method of Gassner and Klinz~\cite{gassner},
Hook~\cite{James} showed that this can be reduced to a $O(n^3d^2)$ time.

The following result generalizes Theorem~\ref{th-major-cont}.
\begin{theorem}\label{th-major-pencil}
Let $\sA$ be as in~\eqref{defAepsilon} and~\eqref{pencil-asymp}
and denote by $A$ the min-plus matrix polynomial~\eqref{defAtrop}
with coefficients $A_k$ as in~\eqref{pencil-asymp}. 
Assume also that $\sA_d=I$, the identity matrix in
$\cont^{n\times n}$ so that $A_d=I$ the identity matrix in $\rmin^{n\times n}$.
Then, $A$ has $nd$ min-plus algebraic eigenvalues (belonging to $\rmin$).
Let  $\Gamma=(\gamma_1\leq\cdots \leq \gamma_{nd})$ be the sequence of 
these eigenvalues.
Assume that the eigenvalues $\sL^1_\eps,\ldots , \sL^{nd}_\eps$
of $\sA_\eps$ (counted with multiplicities) have first order asymptotics,
\begin{align*}%
\sL_\eps^i\sim \lambda_i \eps^{\Lambda_i}\enspace ,
\end{align*}
and let $\Lambda = (\Lambda_1 \leq \cdots \leq \Lambda_n)$.
Then, $\Lambda  \weakm\Gamma$.
\end{theorem}
\begin{proof}
Let us denote by $\sQ=\det({\sA})\in\cont[\iY]$ 
the characteristic polynomial of $\sA$ 
and by $P=\perm A$ the min-plus characteristic polynomial of $A$.
Since $\sA_d=I$, we have $\sQ_{nd}=1$ and $P_{nd}=\unit$.
The eigenvalues $\sL^1_\eps,\ldots , \sL^{nd}_\eps$
of $\sA_\eps$ are the roots of the polynomial $\sQ$.
Assume that they  have first order asymptotics with exponents
$\Lambda_1\leq \cdots \leq \Lambda_{nd}$.
From~\eqref{puis0}, these are elements of $\cont\proper$
and their images by $\dexo$ (counted with multiplicities)
are equal to $\Lambda_1\leq \cdots \leq \Lambda_{nd}$.
By Proposition~\ref{prop-kap-pre} applied to the ring $\cont$,
$\Lambda=(\Lambda_1\leq \cdots \leq \Lambda_{nd})$
coincide with the roots of the min-plus polynomial $Q=\dex{\sQ}$.
Then, computing the coefficients of the formal polynomial $\sQ$,
we can deduce that $Q\geq P$. Since $Q_{nd}=P_{nd}=\unit$,
we get that $\Lambda=R(Q)\weakm R(P)=\Gamma$.

Note that another way to prove this result is to apply a block companion 
transformation to the matrix polynomials $\sA$ and $A$,
yielding linear pencils $\sB$ and $B$ with leading terms
$\sB_1$ and $B_1$ equal to the
identity matrix, then to apply Theorem~\ref{th-major-cont} to 
$\sB_0$.
\end{proof}

For any matrix $B\in \rmin^{n\times n}$ such that $\perm B\neq\zero$,
we define the graph $\opt(B)$ as the set of arcs belonging
to optimal assignments: the set of nodes of $\opt(B)$ 
is $[n]$ and there is an arc from $i$ to
$j$ if there is a permutation $\sigma$ such that
$j=\sigma(i)$ and $|\sigma|_B=\perm B$.
One can compute $\opt(B)$ by  using the following construction.

We shall say that two vectors
$U,V$ of dimension $n$ with entries in $\R=\rmin\setminus\{\zero\}$
form a \new{Hungarian pair} with respect to $B$ if, for all $i,j$,
we have $B_{ij}\geq U_iV_j$, 
and $U_1\cdots U_n V_1\cdots V_n=\perm B$,
the products being understood
in the min-plus sense (since $U,V$ are seen as elements of $\rmin^n$).
Thus, $(U,V)$ coincides with the
optimal dual variable in the linear programming
formulation of the optimal assignment problem.
In particular, a Hungarian pair always exists if
the optimal assignment problem is feasible,
i.e., if $\perm B \neq \zero$, and it can
be computed in $O(n^3)$ time by the Hungarian
algorithm (see for instance~\cite[\S~17]{schrijver}).
For any Hungarian pair $(U,V)$, we now define 
the {\em saturation graph}, 
$\sat(B,U,V)$, which has set of nodes $[n]$
and an arc from $i$ to $j$ if $B_{ij}=U_iV_{j}$.
We shall see in the following section that $\opt(B)$ can be 
computed easily from $\sat(B,U,V)$.

For any min-plus matrix polynomial $A$ and any
scalar $\gamma\in\rmin$, we denote by $\G_k(A,\gamma)$ the graph
with set of nodes $[n]$, and an arc from $i$ to
$j\in [n]$ if $\gamma^k(A_k)_{ij}=\eval{A}_{ij}(\gamma)\neq \zero$.
This is a subgraph of the graph of $\eval{A}(\gamma)$.
For any graphs $G$ and $G'$, the {\em intersection} $G\cap G'$ is
the graph whose set of nodes (resp.\ arcs) is the
intersection of the sets of nodes (resp.\ arcs) of $G$ and $G'$.
Finally, if $G$ is any graph with set of nodes $[n]$, 
and if $b\in \C^{n\times n}$,
we define the matrix $b^{G}$ by
$(b^{G})_{ij}=b_{ij}$ if $(i,j)\in G$,
and $(b^{G})_{ij}=0$ otherwise.
The following results generalize Theorem~\ref{th3-local}
and~\ref{th-gen=}, respectively. 
Their proof will be given in the next section.

\begin{theorem}
\label{th-1}
Let $\sA$ be a regular matrix polynomial over $\cont$
as in~\eqref{defAepsilon} satisfying~\eqref{pencil-asymp}
and denote by $A$ the min-plus matrix polynomial~\eqref{defAtrop}
with coefficients $A_k$ as in~\eqref{pencil-asymp}.
Then $A$ is a regular matrix polynomial and we have:
\begin{enumerate}
\item\label{th-1-val} $\val P_{\sA_\epsilon}\geq \val P_A$, which means that
if $\zero=+\infty$ is an eigenvalue of $A$ with multiplicity
$m_{\zero,A}$, then 
$0$ is an eigenvalue of $\sA_\epsilon$ with a multiplicity 
at least $m_{\zero,A}$.
\item\label{th-1-deg} $\deg P_{\sA_\epsilon}\leq \deg P_A$, which means that
if $-\infty$ is an eigenvalue of $A$ with multiplicity
$m_{-\infty,A}$, then 
$\infty$ is an eigenvalue of $\sA_\epsilon$ with a multiplicity 
at least $m_{-\infty,A}$.
\item\label{th-1-finite} Let $\gamma$ denote any finite ($\neq\pm\infty$) algebraic eigenvalue of $A$, and denote by $m_{\gamma,A}$ its multiplicity.
Let $G$ be equal either to $\opt(\eval A(\gamma))$ or 
$\sat(\eval A(\gamma),U,V)$,
for any choice of the Hungarian pair $(U,V)$ with respect to $\eval A(\gamma)$,
and let $G_k=\G_{k}(A,\gamma)\cap G$ for all $0\leq k\leq d$.
Assume that the matrix polynomial
\begin{align}\label{pencilsat}
a^{(\gamma)}:= a_0^{G_{0}}  + \iY a_1^{G_{1}} +\cdots + 
\iY^da_d^{G_{d}} \enspace .
\end{align}
is regular. 
Let $m_\gamma\geq 0$ be the number of non-zero eigenvalues
of the matrix polynomial $a^{(\gamma)}$ and let
$\lambda_1,\ldots,\lambda_{m_\gamma}$ be these eigenvalues.
Then $m_\gamma\leq m_{\gamma,A}$ and
the matrix polynomial $\sA_\epsilon$ has $m_\gamma$ eigenvalues
$\sL_{\epsilon,1},\ldots,\sL_{\epsilon,m_\gamma}$
with first order asymptotics of the form
\(
\sL_{\epsilon,i}\sim \lambda_i \epsilon^{\gamma}
\).

Let us denote by $m'_{\gamma,A}$ the sum of the multiplicities of all the algebraic  eigenvalues
of $A$ greater than $\gamma$ ($+\infty$ comprised), 
putting $m'_{\gamma,A}=0$ if no such eigenvalues exist,
and let $m''_{\gamma,A}=nd-m_{\gamma,A}-m'_{\gamma,A}$.
Let $m'_\gamma=\val \det (a^{(\gamma)})$ be the multiplicity of
$0$ as an eigenvalue of the matrix polynomial $a^{(\gamma)}$,
with $m'_\gamma=0$ if $0$ is not an eigenvalue.
Let also $m''_\gamma=nd-m_\gamma-m'_\gamma$ be the multiplicity of 
$\infty$  as an eigenvalue of the matrix polynomial $a^{(\gamma)}$,
as it were of degree $d$.
Then $m'_\gamma\geq m'_{\gamma,A}$ (resp.\  $m''_\gamma\geq m''_{\gamma,A}$)
and the matrix polynomial $\sA_\epsilon$ has precisely $m'_\gamma$ (resp.\  
$m''_\gamma$) eigenvalues $\sL_\epsilon$ such that
$\sL_\eps\simeq 0 \epsilon^{\gamma}$ 
(resp.\ $\sL_\eps^{-1}\simeq 0 \epsilon^{-\gamma}$, that is
the modulus of $\epsilon^{-\gamma}\sL_\eps$ converges to infinity).
\end{enumerate}
\end{theorem}
\begin{theorem}
\label{th-1-gen}
Let $\sA$ be a matrix polynomial over $\cont$
as in~\eqref{defAepsilon} satisfying~\eqref{pencil-asymp}
and denote by $A$ the min-plus matrix polynomial~\eqref{defAtrop}
with coefficients $A_k$ as in~\eqref{pencil-asymp}.
Assume that $A$ is a regular matrix polynomial.
Then, for generic values of the parameters $(a_k)_{ij}$,
the matrix polynomial $\sA$ is regular,
the inequalities in Theorem~\ref{th-1} are equalities,
and all eigenvalues of the matrix po lynomial $\sA_\epsilon$ have
first order asymptotics. More precisely,
we have:
\begin{enumerate}
\item $\val P_{\sA_\epsilon}= \val P_A$, which means that
$\zero=+\infty$ is an eigenvalue of $A$ if and only if
$0$ is an eigenvalue of $\sA_\epsilon$ and they have same multiplicity.
\item $\deg P_{\sA_\epsilon}= \deg P_A$, which means that
$-\infty$ is an eigenvalue of $A$ if and only if
$\infty$ is an eigenvalue of $\sA_\epsilon$ and they have same multiplicity.
\item for any finite ($\neq\pm\infty$) algebraic eigenvalue  $\gamma$ of $A$,
the matrix polynomial $a^{(\gamma)}$ of~\eqref{pencilsat} is regular, and
has $m_{\gamma,A}$ non-zero eigenvalues. Moreover, 
$m'_{\gamma,A}=\val \det (a^{(\gamma)})$ is the the multiplicity of
$0$ as an eigenvalue of the matrix polynomial $a^{(\gamma)}$,
and  $m''_{\gamma,A}=nd-m_{\gamma,A}-m'_{\gamma,A}$ is the multiplicity of 
$\infty$  as an eigenvalue of the matrix polynomial $a^{(\gamma)}$,
as it were of degree $d$.
\end{enumerate}
\end{theorem}

\begin{example} \label{example-lidski}
Consider $\sA=\sA_{0}-\iY\id$, with
\[
\sA_{0} = \left[\begin{array}{ccc}
b_{11} \eps & b_{12} & b_{13}\eps\\
b_{21} \eps & b_{22}\eps & b_{23}\\
b_{31} \eps & b_{32}\eps & b_{33}\eps
\end{array}\right]\enspace,\;\mrm{where $b_{ij}\in \C$.}
\]
When $b_{12}=b_{23}=1$, the matrix $\sA_0$  corresponds to 
the perturbation of a Jordan block of size $3$ and zero-eigenvalue.
Vi\v sik, Ljusternik and Lidski\u\i\ theory predicts that
the eigenvalues of $\sA_0$ (thus of $\sA$) have first order 
asymptotics of the form $\sL_{\eps,i}\sim \lambda_i \eps^{1/3}$, $i=1,\ldots, 3$,
where $\lambda_i$ are the roots of $\lambda_i^3=b_{12}b_{23}b_{31}$.
Assume now that $b_{31}=0$, so that we are in a singular case of the 
Vi\v sik, Ljusternik and Lidski\u\i\ theory, and 
let us show that the results of the present section apply.

The associated min-plus matrix polynomial and characteristic polynomial
function are 
\[ A= \left[\begin{array}{ccc}1\oplus \iY&0&1\\1&1\oplus \iY&0\\+\infty&1&1\oplus \iY
\end{array}\right]\enspace,
\qquad 
\widehat{P_A}(x)=(x\oplus 0.5)^2(x\oplus 1)\enspace,
\]
so that the eigenvalues of $A$ are $\gamma_1=\gamma_2=0.5$,
with multiplicity $2$, and $\gamma_3=1$, 
with multiplicity $1$.
Hence, Theorem~\ref{th-1-gen} predicts that two of the
eigenvalues  of $\sA_0$ have first asymptotics of the
form $\sL\sim \lambda \eps^{1/2}$, with $\lambda\neq 0$, and 
that one of them has first asymptotics of the
form $\sL\sim \lambda \eps^{1}$.
 Note that the eigenvalue $\gamma=0.5$ is the unique
geometric min-plus eigenvalue of $\sA_0$ and that the associated critical
graph covers all nodes $\{1,2,3\}$, so that the results 
of~\cite{abg04b} can only predict the two first asymptotics of the
form $\sL\sim \lambda \eps^{1/2}$.

Let us now detail the results of Theorem~\ref{th-1-gen}.
We first consider the eigenvalue
$\gamma=0.5$.
Then $U=(0,0.5,1)$ and $V=(0.5, 0, -0.5)$ yields
a Hungarian pair with respect to the matrix
\[
\eval A(0.5)= \left[\begin{array}{ccc}0.5_{1}&0_0&1
\\ 1_0&0.5_1&0_0\\ +\infty&1_0& 0.5_1
\end{array}\right]\enspace,
\]
where we adopt the following convention
to visualize the graphs $G_k=\G_{k}(A,\gamma)\cap \sat(\eval A(\gamma),U,V)$:
an arc $(i,j)$ belongs to $G_k$ if $k$ is put
as a subscript of the entry $\eval A_{ij}(\gamma)$.
For instance, $\eval{A}_{11}(0.5)=0.5$, and $(1,1)$
belongs $G_1$.
Entries without subscripts, like
$\eval A_{13}(0.5)=1$, correspond
to arcs which do not belong to $\sat(\eval A(0.5),U,V)$.
The eigenvalues of the matrix polynomial $a^{(0)}$ are the roots of
\[
\det \left[\begin{array}{ccc}
-\lambda & b_{12} & 0\\
b_{21} & -\lambda &b_{23} \\
0 & b_{32} & -\lambda
\end{array}\right] 
= \lambda(-\lambda^2+ b_{12}b_{21}+b_{32}b_{23}) =0
\enspace.
\]
Theorem~\ref{th-1-gen} predicts that this equation has,
for generic values of the parameters $b_{ij}$,
two non-zero roots, $\lambda_1,\lambda_2$,
which yields two eigenvalues of $\sA_\epsilon$,
$\sL_{\epsilon,i}\sim \lambda_i\epsilon^{1/2}$,
for $i=1,2$. Here $\lambda_1$ and $\lambda_2$ are the square roots
of $b_{12}b_{21}+b_{32}b_{23}$.

Consider finally the eigenvalue $\gamma=1$.
We can take $U=(0,0,1)$, $V=(1,0,0)$,
and the previous computations become
\[
\eval A(1)= \left[\begin{array}{ccc}1_{01}&0_0&1
\\ 1_0&1&0_0\\
+\infty& 1_0&1_{01}
\end{array}\right]\enspace,\]
\[\det
\left[\begin{array}{ccc}
b_{11} -\lambda & b_{12} & 0\\
b_{21} & 0&b_{23} \\
0 & b_{32} & b_{33}-\lambda 
\end{array}\right] =
\lambda(b_{12}b_{21}+b_{23}b_{32})-
(b_{12}b_{21}b_{33}+b_{11}b_{23}b_{32})
=0 \enspace .
\]
Theorem~\ref{th-1-gen} predicts that this equation has,
for generic values of the parameters $b_{ij}$,
a unique nonzero root, $\lambda_3$, 
and that there is a branch $\sL_{\epsilon,3}\sim \lambda_3 \epsilon$.
Here $\lambda_3=(b_{12}b_{21}b_{33}+b_{11}b_{23}b_{32})/(b_{12}b_{21}+b_{23}b_{32})$.
\end{example}

\section{Preliminaries on Hungarian pairs}
\label{sec-hung}

To prove Theorems~\ref{th-1} and~\ref{th-1-gen}, we need to
establish some properties of the saturation graph associated to
Hungarian pairs.

We shall adopt the following notation.
For any $U\in\R^n$, we denote by $\diag{U}$ the diagonal
$n\times n$ matrix over $\rmin$ such that
$(\diag{U})_{ii}=U_i$, so that $(\diag{U})_{ij}=+\infty$ for $i\neq j$.
For any permutation $\sigma$ of $[n]$,
we denote by $\Psigmamin$ its associated min-plus permutation matrix:
$(\Psigmamin)_{ij}=\unit$ if $j=\sigma(i)$
and $(\Psigmamin)_{ij}=\zero$ otherwise. We reserve
the simpler notation $P^\sigma$ for the ordinary permutation matrix 
which is such that $P^{\sigma}_{ij}=1$ if $j=\sigma(i)$
and $P^{\sigma}_{ij}=0$ otherwise.

Moreover, if $U\in\R^n$, we use the notation 
$U_\sigma:=(U_{\sigma(i)})_{i=1,\ldots, n}$.
We shall say that a matrix $M\in \rmin^{n\times n}$ is \new{monomial}
if it can be written as $M=\diag{U}\Psigmamin$ for some $U\in\R^n$
and $\sigma\in\Perm_n$.
We have equivalently $M=\Psigmamin\diag{V}$, by taking
$V=U_{\sigma^{-1}}$.
Recall that the monomial matrices are the only invertible matrices over $\rmin$,
and that $M^{-1}=\Psigmainvmin\diag{-U}$.
For all matrices $A$,  $A^T$ denotes the transpose of $A$.
Now if $G$ is a graph with set of nodes $[n]$,
and $\sigma,\tau\in\Perm_n$, 
we denote by $G_{\sigma,\tau}$ the graph with the same set of nodes, and
an arc $(i,j)$ if and only if $(\sigma(i),\tau^{-1}(j))$ is an arc of $G$.

The following elementary result will allows one to normalize matrices
in a suitable way. 
\begin{lemma}\label{lem-permutation}
Let $B\in \rmin^{n\times n}$ such that $\perm B\neq\zero$,
and let $M=\diag{W}\Psigmamin$ and
$N=\diag{X}\Ptaumin$ be monomial matrices, with $W,X\in\R^n$
and $\sigma,\tau\in\Perm_n$.
Then $\perm (MBN)= W_1\cdots W_n X_1\cdots X_n \perm B=(\perm M) (\perm B)
(\perm N)$,
$\nu$ is an optimal permutation for $MBN$ if and only if
$\tau^{-1}\circ \nu\circ \sigma^{-1}$ is an optimal permutation for $B$, and
 $\opt(MBN)=\opt(B)_{\sigma,\tau}$.

Let  $(U,V)$ be a Hungarian pair with respect to $B$.
Then, $(MU,N^TV)$ is  a Hungarian pair with respect to $MBN$,
and we have $\sat(MBN,MU,N^TV)=\sat(B,U,V)_{\sigma,\tau}$.
\end{lemma}
\begin{proof}
Let $B$, $M$, and $N$ be as in the lemma.
For all $i,j\in [n]$, we have $(MBN)_{ij}= W_i B_{\sigma(i),\tau^{-1}(j)}
 X_{\tau^{-1}(j)}$.
Hence, for all $\nu\in\Perm_n$, the weight of $\nu$ with respect to
$MBN$ is equal to 
\[|\nu|_{MBN}=W_1\cdots W_n X_1\cdots X_n |\tau^{-1}\circ \nu\circ \sigma^{-1}|_B
\enspace. \]
Hence,  $\perm (MBN)=W_1\cdots W_n X_1\cdots X_n \perm B$,
and $\nu$ is an optimal permutation for $MBN$ if and only if
$\tau^{-1}\circ \nu\circ \sigma^{-1}$ is an optimal permutation for $B$.
Moreover, since $\perm M= W_1\cdots W_n $ and $\perm N= X_1\cdots X_n $,
we deduce that  $\perm (MBN)=(\perm M)(\perm B) (\perm N)$.

Let $(i,j)$ be an arc of $\opt(MBN)$. Then, by definition,
 there exists an optimal permutation $\nu$ for $MBN$
such that $j=\nu(i)$, hence $\nu'=\tau^{-1}\circ \nu\circ \sigma^{-1}$
is  an optimal permutation  for $B$ such that $\tau^{-1}(j)=\nu'(\sigma(i))$,
which implies that $(\sigma(i),\tau^{-1}(j))$ is an arc of $\opt(B)$,
so $(i,j)$ is an arc of $\opt(B)_{\sigma,\tau}$.
Since the reverse implication is also true (replace $M$ and $N$ 
by their inverse matrices), 
this shows the first assertion of the lemma.

Now let $(U,V)$ be a Hungarian pair with respect to $B$.
We have $B_{ij}\geq U_{i} V_{j}$ for all $i,j\in [n]$.
Hence, $(MBN)_{ij}= W_i B_{\sigma(i),\tau^{-1}(j)}
 X_{\tau^{-1}(j)} \geq W_i U_{\sigma(i)} V_{\tau^{-1}(j)}  X_{\tau^{-1}(j)} =
(MU)_{i} (N^TV)_j$, and
\begin{align*}
\perm (MBN)&=W_1\cdots W_n X_1\cdots X_n \perm B\\
&= W_1\cdots W_n U_1\cdots U_nX_1\cdots X_nV_1\cdots V_n\\
&= (MU)_1\cdots (MU)_n(N^TV)_1\cdots (N^TV)_n\enspace,
\end{align*}
so that $(MU,N^TV)$ is  a Hungarian pair with respect to $MBN$.
Finally, $(i,j)$ is an arc of $\sat(MBN,MU,N^TV)$
if and only if $(MBN)_{ij}=(MU)_{i} (N^TV)_j$,
which is equivalent to $B_{\sigma(i) ,\tau^{-1}(j)}= U_{\sigma(i)} V_{\tau^{-1}(j)}$,
then to the property that
$(\sigma(i),\tau^{-1}(j))$ is an arc of $\sat(B,U,V)$,
hence to the one that $(i,j)$ is an arc of $\sat(B,U,V)_{\sigma,\tau}$.
\end{proof}

We denote by $\unit$ the vector of $\R^n$ with all its entries
equal to $\unit=0$.

\begin{corollary}\label{coro1}
Let $B\in \rmin^{n\times n}$ be such that $\perm B\neq\zero$,
and let $(U,V)$ be a Hungarian pair with respect to $B$.
Then, $(\unit,\unit)$ is a Hungarian pair with respect to 
the matrix $C:=\diag{U}^{-1}B\diag{V}^{-1}\in\rmin^{n\times n}$,
$\sat(B,U,V)=\sat(C,\unit,\unit)$  and $\opt(B)=\opt(C)$.
In particular, $\perm C=\unit$, 
$C_{ij}\geq \unit=0$ for all $i,j\in [n]$ and $(i,j)$ is an arc
of $\sat(B,U,V)$ if and only if $C_{ij}=\unit$.
\end{corollary}
\begin{proof}
Let $B,U,V,C$ be as in the corollary.
Then,  by Lemma~\ref{lem-permutation}
applied to $M=\diag{U}^{-1}$ and $N=\diag{V}^{-1}$, we obtain the
first assertion of the corollary, together with $\perm C=\unit$.
Then, $C_{ij}\geq \unit=0$ for all $i,j\in [n]$ and $(i,j)$ is an arc
of $\sat(B,U,V)$ if and only if $C_{ij}=\unit$.
\end{proof}

\begin{corollary}\label{coro2}
Let $B\in \rmin^{n\times n}$ be such that $\perm B\neq\zero$,
and let $(U,V)$ be a Hungarian pair with respect to $B$.
Then, $\opt(B)\subset\sat(B,U,V)$.
Moreover, the following are equivalent for $\sigma\in \Sym_n$:
(i) $\sigma$ is a permutation of  $\sat(B,U,V)$; (ii)
$\sigma$ is a permutation of $\opt(B)$; (iii) $\sigma$
is an optimal permutation of $B$.
\end{corollary}
\begin{proof}
Let $C$ be as in Corollary~\ref{coro1}.
Then, if $(i,j)$ is an arc of $\opt(C)$,
there exists $\sigma\in\Perm_n$ such that
$j=\sigma(i)$ and $|\sigma|_C=\perm C=\unit$.
Since all entries of $C$ are greater or equal to $\unit$, 
this implies that $C_{ij}=\unit$, and so
$(i,j)$ is an arc of $\sat(C,\unit,\unit)$.
Since $\opt(C)=\opt(B)$ and $\sat(C,\unit,\unit)=\sat(B,U,V)$,
this shows that $\opt(B)\subset\sat(B,U,V)$.

Let us show that for $\sigma\in \Sym_n$, (i)$\Rightarrow$(iii)$\Rightarrow$(ii)$\Rightarrow$(i).
Let $\sigma$ be a permutation of  $\sat(B,U,V)$, this means that
$(i,\sigma(i))$ is an arc of $\sat(B,U,V)$, for all $i \in [n]$.
Then  $C_{i\sigma(i)}=\unit$, for all  $i \in [n]$, which implies that
 $|\sigma|_C=\unit=\perm C$, and so $\sigma$ is an optimal permutation 
of $C$, hence an optimal permutation of $B$, by Lemma~\ref{lem-permutation},
which shows (i)$\Rightarrow$(iii).
By definition of $\opt(B)$, if $\sigma$ be an optimal permutation of $B$,
then $(i,\sigma(i))$ is an arc of $\opt(B)$, for all $i\in [n]$,
so $\sigma$ is a permutation of $\opt(B)$, which shows (iii)$\Rightarrow$(ii).
If now $\sigma$ be a permutation of  $\opt(B)$, then, for all $i\in [n]$,
$(i,\sigma(i))$ is an arc of $\opt(B)\subset \sat(B,U,V)$,  hence, 
$\sigma$ is a permutation of  $\sat(B,U,V)$, which shows (ii)$\Rightarrow$(i).
\end{proof}

\begin{corollary}\label{coro3}
Let $B\in \rmin^{n\times n}$ such that $\perm B\neq\zero$,
let $\sigma$ be an optimal permutation for $B$,
and let $(U,V)$ be  a Hungarian pair with respect to $B$.
Then,  the identity map is an optimal permutation for 
$\Psigmainvmin B$ and
 $(U_{\sigma^{-1}},V)$ is  a Hungarian pair with respect to $\Psigmainvmin B$.
Moreover, we have $\perm (\Psigmainvmin B)=\perm (B)$,
$\opt(\Psigmainvmin B)=\opt(B)_{\sigma^{-1},\idp}$,
and $\sat(\Psigmainvmin B,U_{\sigma^{-1}},V)=\sat(B,U,V)_{\sigma^{-1},\idp}$.
\end{corollary}
\begin{proof}
From Lemma~\ref{lem-permutation},
$\nu$ is an optimal permutation for $\Psigmainvmin B$ if and only if
$\nu\circ \sigma$ is an optimal permutation for $B$.
Hence if $\sigma$ is an optimal permutation for $B$,
 the identity map is an optimal permutation for $\Psigmainvmin B$.
The other properties follow from Lemma~\ref{lem-permutation}.
\end{proof}

\begin{proposition}\label{prop-carac-hungarian}
Let $B\in \rmin^{n\times n}$ such that $\perm B=\unit$,
$(\unit,\unit)$ is a Hungarian pair with respect to $B$ and
the identity map is an optimal permutation of $B$.
Then $B\unit =\unit$, $\rhomin(B)=\unit$, and $\opt(B)$
is the critical graph of $B$.
\end{proposition}
\begin{proof}
Let $B$ be as in the proposition.
Since $(\unit,\unit)$ is a Hungarian pair with respect to $B$,
Corollary~\ref{coro1} shows that
all entries of $B$ are greater or equal to $\unit$, and
that $B_{ij}=\unit$ for all arcs $(i,j)$ of $\sat(B,\unit,\unit)$.
Then, by Corollary~\ref{coro2}, $B_{ij}=\unit$ for all arcs $(i,j)$ of 
an optimal permutation of $B$.
Since the identity map is an optimal permutation of $B$,
this implies that $B_{ii}=\unit$ for all $i$,
which implies $B\unit=\unit$.
Moreover, all circuits have a weight with respect to $B$ greater or equal
to $\unit$ and the circuits of an optimal permutation
have a weight equal to $\unit$. This implies that $\rhomin(B)=\unit$.

Let  $(i,j)$ be an arc of the critical graph of $B$.
By definition (see Section~\ref{sec-tropeig}),
there exists a circuit passing through $(i,j)$ 
with mean weight with respect to $B$ equal to $\unit$, 
so with weight equal to $\unit$.
Since the identity map is an optimal permutation of $B$,
any circuit  with weight $\unit$ can be completed into a permutation 
with weight $\unit$, by taking the identity on the complementary of this 
circuit, hence, $(i,j)$ is an arc of  $\opt(B)$.
Conversely, let  $(i,j)$ be an arc of  $\opt(B)$, then there exists
an optimal permutation $\sigma$ for $B$ such that $j=\sigma(i)$.
This implies (as above) that $B_{k\sigma(k)}=\unit$, for all $k\in [n]$.
Then, all the circuits of $\sigma$ have a weight equal to $\unit$,
and so are critical circuits. Since $(i,j)$ belongs to one of them, this shows
that $(i,j)$ is an arc of the critical graph of $B$.
\end{proof}

\begin{corollary}\label{coro4}
Let $B\in \rmin^{n\times n}$ be as in Proposition~\ref{prop-carac-hungarian}. 
Then, $\opt(B)$ is the disjoint union of the strongly connected components
of $\sat(B,\unit,\unit)$.
\end{corollary}
\begin{proof}
By Proposition~\ref{prop-carac-hungarian}, $\unit$ is a fixed point 
of $B$, and $\opt(B)$ is the critical graph of $B$.
It is known that the critical graph of $B$ is
the disjoint union of the strongly connected components of
the saturation graph of $B$ on any eigenvector $U$ of $B$, that is the
set of arcs $(i,j)$ such that $B_{ij} U_j=U_i$
(see for instance~\cite{abg04b}).
When $U=\unit$, this saturation graph coincides with $\sat(B,\unit,\unit)$,
which shows the corollary.
Let us reproduce the proof in that case.
Indeed, let $(i,j)$ be an arc of a strongly connected component of
the saturation graph of $B$ on the eigenvector $\unit$. Then there exists
a path from $j$ to $i$ in this saturation graph.
Concatenating this path with the arc $(i,j)$, we get a circuit,
the weights with respect to $B$ of which are all equal to $\unit$, hence a 
critical circuit of $B$. This shows that $(i,j)$ is an arc of the critical graph
of $B$. Since the critical graph of $B$ is the
union of its strongly connected components, and is always included
in the saturation graph, we get the desired property.
\end{proof}

\begin{corollary}\label{coro5}
Let $B\in \rmin^{n\times n}$ such that $\perm B\neq\zero$,
$(U,V)$ be a Hungarian pair with respect to $B$,
and suppose that the identity map is an optimal permutation of $B$.
Then, $\opt(B)$ is the disjoint union of the strongly connected components
of $\sat(B,U,V)$.
\end{corollary}
\begin{proof}
Corollary~\ref{coro1} shows that the matrix $C$ defined there satisfies
$\perm C=\unit$, that
$(\unit,\unit)$ is a Hungarian pair with respect to $C$, together with
$\opt(B)=\opt(C)$ and $\sat(B,U,V)=\sat(C,\unit,\unit)$.
Moreover, by Lemma~\ref{lem-permutation}, the optimal
permutations for $B$ and $C$ coincide, so that 
the identity map is an optimal permutation of $C$.
Hence, $C$ satisfies the assumptions of Proposition~\ref{prop-carac-hungarian},
so that, by Corollary~\ref{coro4}, $\opt(C)$ 
is the disjoint union of the strongly connected components
of $\sat(C,\unit,\unit)$. With the above equalities, 
this shows the corollary.
\end{proof}

The above results show that $\opt(B)$ can be easily be constructed from
any Hungarian pair with respect to $B$, as follows.

\begin{corollary}
Let $B\in \rmin^{n\times n}$ such that $\perm B\neq\zero$.
Let $(U,V)$ be a Hungarian pair with respect to $B$,
and $\sigma$ be a permutation of $\sat(B,U,V)$.
Then, $\opt(B)_{\sigma^{-1},\idp}$ is the disjoint union of the 
strongly connected components of $\sat(B,U,V)_{\sigma^{-1},\idp}$.
\end{corollary}

\section{Proof of Theorems~\ref{th-1} and~\ref{th-1-gen}}

We prove here the different assertions in
Theorems~\ref{th-1} and~\ref{th-1-gen}.
Since the intersection of generic sets is a generic set, 
it suffices to prove separately
that each point of Theorem~\ref{th-1-gen} 
holds for generic values of the parameters $(a_k)_{ij}$.

Let us introduce some additional notation.
For any $n\geq 1$ and $d\geq 0$,
the characteristic polynomial  $\det(a)$ of the
matrix polynomial $a=a_0+\iY a_1+\cdots+\iY^d a_d$ with coefficients
in $\C^{n\times n}$ can be thought of as a (formal) complex 
polynomial, denoted $\ch$, in the variables
$\iY$ and $(a_k)_{ij}$ with $k=0,\ldots, d$ and $i,j=1,\ldots , n$,
the coefficients of which are integers.
Hence, the coefficient $\ch_\ell$, $0\leq \ell\leq nd$,
of $\iY^\ell$ in $\ch$ can be thought of as a (formal) complex 
polynomial in the variables
$(a_k)_{ij}$ with $k=0,\ldots, d$ and $i,j=1,\ldots , n$.
In the sequel, we shall write $p(a)$ 
when $p$ is a polynomial in the
variables $(a_k)_{ij}\in\C$ with $k=0,\ldots, d$ and $i,j=1,\ldots , n$,
and  $a=a_0+\iY a_1+\cdots+\iY^d a_d$.
We have in particular $\ch_\ell(a_0+\iY I)=\tr_{n-\ell}(a_0)$.
Similarly, the characteristic polynomial  $\perm a$ of the min-plus
(formal) matrix polynomial $a=a_0\oplus\iY a_1\oplus\cdots\oplus\iY^d a_d$,
with coefficients in $\rmin^{n\times n}$ can be thought of as a (formal) min-plus
 polynomial, denoted  $\ch^{\min}$, in the variables
$\iY$ and $(a_k)_{ij}$ with $k=0,\ldots, d$ and $i,j=1,\ldots , n$,
the coefficients of which are all equal to $\unit$, and 
the coefficient $\ch^{\min}_\ell$, $0\leq \ell\leq nd$,
of $\iY^\ell$ in $\ch^{\min}$ can be thought of as a (formal) min-plus
polynomial in the variables
$(a_k)_{ij}$ with $k=0,\ldots, d$ and $i,j=1,\ldots , n$.
Again, we shall write  $p(a)$ instead of $p$, 
when  $p$ is a min-plus formal polynomial in the
variables $(a_k)_{ij}\in\rmin$ with $k=0,\ldots, d$ and $i,j=1,\ldots , n$,
and $a=a_0\oplus\iY a_1\oplus\cdots\oplus\iY^d a_d$.
We shall also use the notation $\eval{p}(a)$  for the associated 
polynomial function.
We have in particular $\ch^{\min}_\ell(a_0\oplus\iY I)=\trm_{n-\ell}(a_0)$,
where $\trm_k$ is the min-plus (formal) $k$-th trace given by
the formula~\eqref{trmin}.
In the sequel, $\evalch{\ell}$ will denote the polynomial function associated
to the formal polynomial  $\ch^{\min}_\ell$.

In all the proofs of the present section,
we consider a matrix polynomial $\sA$ over $\cont$
as in~\eqref{defAepsilon} satisfying~\eqref{pencil-asymp}
and denote by $A$ the min-plus matrix polynomial~\eqref{defAtrop}
with coefficients $A_k$ as in~\eqref{pencil-asymp}.

\subsection{Proof of the first properties in
Theorems~\ref{th-1} and~\ref{th-1-gen}}
Let us prove the  following implications which correspond to the
first properties stated in Theorems~\ref{th-1} and~\ref{th-1-gen}.
\begin{enumerate}
\item\label{impl1} If $\sA$ is regular, then so does $A$.
\item\label{impl2} If $A$ is regular, then 
for generic values of the parameters $(a_k)_{ij}$,
the matrix polynomial $\sA$ is regular.
\end{enumerate}
From the above notations, 
it is easy to see that the min-plus polynomial $\dex{\ch}$ 
(defined as in~\eqref{qQ}) is equal to $\ch^{\min}$ and
that $\dex{\ch_\ell}=\ch^{\min}_\ell$.
Then, by Lemma~\ref{gener}, we obtain that for $a$, $A$ and $\sA$ as
above,
we have  $\ch_\ell(\sA)\simeq (\ch_\ell)^\sat_A(a)\eps^{\evalch{\ell}(A)}$.
Moreover,  by Lemma~\ref{gener} again, 
for any fixed $A$ such that $\evalch{\ell}(A)\neq \zero$,
we have $(\ch_\ell)^\sat_A(a)\neq 0$
for generic values of $a$ (so of the $(a_k)_{ij}$),
hence the equivalence  $\ch_\ell(\sA)\sim (\ch_\ell)^\sat_A(a)\eps^{\evalch{\ell}(A)}$.

Using these properties, we  see that if $A$ is singular,
then $\evalch{\ell}(A)=\zero$ for all $\ell\geq 0$, hence
$\ch_\ell(\sA)\simeq (\ch_\ell)^\sat_A(a)\eps^{+\infty}$ and so $\ch_\ell(\sA)=0$,
which implies that $\det(\sA)=\ch(\sA)=0$ and so $\sA$ is singular.
Conversely, if $A$ is regular, then there exists $\ell\geq 0$ 
(for instance the degree of $\perm A$) such that
$\alpha=\evalch{\ell}(A)\neq \zero$, and since for generic values of $a$,
$(\ch_\ell)^\sat_A(a)\neq 0$, 
and $\ch_\ell(\sA)\sim (\ch_\ell)^\sat_A(a)\eps^{\alpha}$,
we get that $\ch_\ell(\sA)\neq 0$.
Hence, $\det(\sA)=\ch(\sA)$ is non identically zero, 
so the matrix polynomial $\sA$ is regular for generic values of $a$.

\subsection{Proof of Points~\eqref{th-1-val} and~\eqref{th-1-deg}  of
 Theorems~\ref{th-1} and~\ref{th-1-gen}}

Assume first that $\sA$ is regular as in Theorem~\ref{th-1}.
By the property~\ref{impl1}  already proved in the previous section,
$A$ is also regular. Then, $P_A=\perm A$ and $P_{\sA_\epsilon}=\det (\sA_\epsilon)$ 
have finite valuations and degree,
such that $0\leq \val P_A\leq \deg P_A\leq nd$
and $0\leq \val P_{\sA_\epsilon}\leq \deg P_{\sA_\epsilon}\leq nd$.
The inequalities $\val P_{\sA_\epsilon}\geq \val P_A$ and 
$\deg P_{\sA_\epsilon}\leq \deg P_A$ are trivial when $\val P_A=0$ and
$\deg P_A=nd$.
When $\val P_A>0$, 
$\zero$ is an eigenvalue of $A$ with multiplicity $m_{\zero,A}=\val P_A$,
and $(P_A)_{\ell}=\evalch{\ell}(A)=\zero$, for $\ell<\val P_A$.
Then, by the same arguments as in the previous section,
we have $\ch_\ell(\sA)=0$, for $\ell<\val P_A$,
hence, for all $\eps>0$,  $\val P_{\sA_\epsilon}\geq \val P_A$, and 
$0$ is an eigenvalue of $\sA_\eps$ with multiplicity $\geq m_{\zero,A}$.
Similarly, when $\deg P_A<nd$, $-\infty$ is an 
eigenvalue of $A$ with multiplicity $m_{-\infty,A}=nd-d_A$,
and $(P_A)_{\ell}=\evalch{\ell}(A)=\zero$, for $\ell>\deg P_A$.
Then, by the same arguments as in the previous section,
we have $\ch_\ell(\sA)=0$, for $\ell>\deg P_A$,
hence  $\deg P_{\sA_\epsilon}\leq \deg P_A$,
and $\infty$ is an eigenvalue of $\sA_\eps$ with multiplicity $\geq m_{-\infty,A}$.
This proves Points~\eqref{th-1-val} and~\eqref{th-1-deg}  of 
 Theorem~\ref{th-1}.

Assume now that $A$ is regular as in Theorem~\ref{th-1-gen}.
By the property~\ref{impl2}  already proved in the previous section,
$\sA$ is regular, for generic values of the parameters $(a_k)_{ij}$.
Then, by what is already proved above, for
these generic values of the parameters $(a_k)_{ij}$,
the inequalities $\val P_{\sA_\epsilon}\geq \val P_A$ and 
$\deg P_{\sA_\epsilon}\leq \deg P_A$ hold.
By the definition of the valuation and degree, we have
$(P_A)_{\ell}=\evalch{\ell}(A)\neq \zero$, for
$\ell= \val P_A$ and $\ell=\deg P_A$.
Moreover, using Lemma~\ref{gener}, for both values of $\ell$, we have 
$\ch_\ell(\sA)\simeq (\ch_\ell)^\sat_A(a)\eps^{\evalch{\ell}(A)}$,
and for generic values of $a$, we have $(\ch_\ell)^\sat_A(a)\neq 0$,
and  $\ch_\ell(\sA)\sim (\ch_\ell)^\sat_A(a)\eps^{\alpha}$,
so $\ch_\ell(\sA)\neq 0$.
Hence,  for these generic values of the parameters $(a_k)_{ij}$,
we have  $\val P_{\sA_\epsilon}= \val P_A$ and 
$\deg P_{\sA_\epsilon}=\deg P_A$.
Since the intersection of generic sets is a generic set, 
all the above properties hold for generic values of $a$:
$\sA$ is regular, $\val P_{\sA_\epsilon}= \val P_A$ and  
$\deg P_{\sA_\epsilon}=\deg P_A$.

\subsection{Proof of Point~\eqref{th-1-finite} of
 Theorems~\ref{th-1} and~\ref{th-1-gen}}
Let $\gamma, \; m_{\gamma,A}, \; m'_{\gamma,A}, \; m''_{\gamma,A}$  be as in
Point~\eqref{th-1-finite} of  Theorems~\ref{th-1}.
Let $(U,V)$ be a Hungarian pair $(U,V)$ with respect to $\eval A(\gamma)$,
and let us first consider the case where $G=\sat(\eval A(\gamma),U,V)$,
and $G_k=\G_{k}(A,\gamma)\cap G$ for all $0\leq k\leq d$.

Let us add some notations.
For any $U\in\R^n$, we denote 
by $\diagp{U}$ the diagonal $n\times n$ matrix over $\cont$ such that
$(\diagp{U})_{ii}=\epsilon^{U_i}$ (and $(\diagp{U})_{ij}=0$ for $i\neq j$).
Then $\diagp{U}$ is invertible (in $\cont^{n\times n}$)
and $\diagp{U}^{-1}=\diagp{-U}$.

Consider the scaled matrix polynomial
\[ \sB_\epsilon= \diagp{-U} \sA_\epsilon(\epsilon^{\gamma} \iY)
\diagp{-V} \enspace ,\]
or, making explicit the coefficients of $\sB_\epsilon$,
\[
\sB_\epsilon = \sB_{\epsilon,0}+ \iY \sB_{\epsilon,1} +\cdots+\iY^d\sB_{\epsilon,d}  \enspace .
\]
Then, using for instance Lemma~\ref{gener}, we get
that for every $k=0,\ldots, d$, 
\begin{equation}\label{pencil-asymp-b}
(\sB_{\eps,k})_{ij} \simeq (a_k)_{ij}\epsilon^{(B_k)_{ij}}\enspace,
\qquad \mrm{ for all }  1\leq i,j\leq n \enspace ,
\end{equation}
where $B_k=((B_k)_{ij})\in\rmin^{n\times n}$ is such that
$B_k=\diag{-U} \gamma^k A_k \diag{-V}
=\diag{U}^{-1} \gamma^k A_k \diag{V}^{-1}$.
Let us denote by $B$ the min-plus matrix polynomial
$B=B_0\oplus \cdots \oplus \iY^d B_d$.
Then $\unit$ is an eigenvalue of $B$ with multiplicity $m_{\unit,B}=m_{\gamma,A}$,
and using the appropriate notations, we have
$m'_{\unit,B}=m'_{\gamma,A}$ and $m''_{\unit,B}=m''_{\gamma,A}$.
Moreover, $\eval{B}(\unit)=\diag{U}^{-1}\eval{A}(\gamma) \diag{V}^{-1}$,
and since $(U,V)$ is a  Hungarian pair  with respect to $\eval A(\gamma)$,
we get, by Corollary~\ref{coro1}, that 
$(\unit,\unit)$ is a   Hungarian pair  with respect to $\eval B(\unit)$,
that $\perm \eval B(\unit)=\unit$, 
and $G=\sat(\eval A(\gamma),U,V)=\sat(\eval{B}(\unit),\unit,\unit)$.
We also have $ \G_{k}(A,\gamma)=\G_{k}(B,\unit)$.
So we are reduced to the case where $\gamma=\unit$, and
$U=V=\unit$.
Moreover, $(B_k)_{ij}\geq \eval{B_{ij}}(\unit)\geq \unit$ 
for all $k=0,\ldots, d$ and $i,j\in [n]$,
and we have $(B_k)_{ij}= \unit$ if and only if $(i,j)\in G_k=\G_{k}(B,\unit)
\cap G$. Using~\eqref{pencil-asymp-b}, this implies that 
\[ \lim_{\epsilon\to 0} \sB_\epsilon= a^{(\gamma)}\enspace,\]
where $a^{(\gamma)}$ is given in~\eqref{pencilsat}.
This implies that 
$\lim_{\epsilon\to 0} \det( \sB_\epsilon)= \det (a^{(\gamma)})$.

Using the above notations,
we have that the formal characteristic polynomials of
$\sB$ and $B$ are given by 
$\det (\sB_\eps)=\sum_{\ell=0}^{nd}\ch_{\ell}(\sB_\eps)\iY^\ell$
and $\perm B=\oplus_{\ell=0}^{nd}\evalch{\ell}(B)\iY^\ell$.
Moreover, by Lemma~\ref{gener},  we have
$\ch_\ell(\sB_\eps)\simeq (\ch_\ell)^\sat_B(a)\eps^{\evalch{\ell}(B)}$.
Hence the polynomial $\sP=\det (\sB)\in\cont[\iY]$
satisfies the assumptions of Theorem~\ref{th3-local} 
with the degree $nd$ instead of $n$, $P=\perm B$,
and $p=\sum_{\ell=0}^{nd}(\ch_\ell)^\sat_B(a)$.
Moreover the scalar $c=\unit$ is a finite root of $P$ with multiplicity 
$m=m_{\gamma,A}$, and we have $m'=m'_{\gamma,A}$ in Theorem~\ref{th3-local}.
Since $\eval{P}(\unit)=\perm \eval B(\unit)=\unit$, we get that
$p^{(\unit)}=\lim_{\epsilon\to 0}\det( \sB_\epsilon)$ and so the above 
properties show that $p^{(\unit)}= \det (a^{(\gamma)})$.

Assume now that the matrix polynomial  $a^{(\gamma)}$ is regular.
This means that  $p^{(\unit)}$ is non identically zero.
Applying Theorem~\ref{th3-local}, we get that if it has $m_\gamma\geq 1$
non-zero eigenvalues, $\lambda_1,\ldots,\lambda_{m_\gamma}$,
then, the matrix polynomial $\sB_\epsilon$ has $m_\gamma$ eigenvalues
$\sM_{\epsilon,1},\ldots,\sM_{\epsilon,m_\gamma}$
with limits $\lambda_1,\ldots,\lambda_{m_\gamma}$. This implies that 
 $\sA_\epsilon$ has $m_\gamma$ eigenvalues
$\sL_{\epsilon,i}=\epsilon^{\gamma} \sM_{\epsilon,i}$, $i=1,\ldots,m_\gamma$,
with first order asymptotics of the form
\(
\sL_{\epsilon,i}\sim \lambda_i \epsilon^{\gamma}
\).
If $0$ is an eigenvalue of the matrix polynomial $a^{(\gamma)}$
with multiplicity $m'_\gamma$, then the valuation of
$p^{(\unit)}$ is equal to $m'_\gamma$, hence applying Theorem~\ref{th3-local}, 
we get that the matrix polynomial $\sB_\epsilon$ 
has precisely $m'_\gamma$ eigenvalues converging to $0$.
Finally, the remaining $m''_\gamma=nd-m_\gamma-m'_\gamma$ 
eigenvalues of $\sB_\epsilon$ have a modulus 
converging to infinity. Hence, 
the matrix polynomial $\sA_\epsilon$ has precisely $m'_\gamma$
eigenvalues $\sL_\epsilon$ such that
$\sL_\eps\simeq 0 \epsilon^{\gamma}$, and $m''_\gamma=nd-m_\gamma-m'_\gamma$ 
eigenvalues $\sL_\epsilon$ 
such that the modulus of $\epsilon^{-\gamma}\sL_\eps$
converges to infinity.
Finally the inequalities $m_\gamma\leq m_{\gamma,A}$,
$m'_\gamma\geq m'_{\gamma,A}$ and $m''_\gamma\geq m''_{\gamma,A}$ are deduced 
from Theorem~\ref{th3-local}.
This finishes the
proof of Point~\eqref{th-1-finite} of  Theorem~\ref{th-1} 
in the case $G=\sat(\eval A(\gamma),U,V)$.

Let us show Point~\eqref{th-1-finite} of  Theorem~\ref{th-1-gen}.
Let us fix the matrix polynomial $A$ and so the matrix polynomial $B$.
By Lemma~\ref{gener},  for all $\ell$ such that 
$P_{\ell}=\evalch{\ell}(B)\neq \zero$, we have that,
for generic values of $a$, $p_{\ell}=(\ch_\ell)^\sat_B(a)\neq 0$.
Since $\unit$ is a finite root of $P$ with multiplicity $m_{\gamma,A}$,
and $m'_{\gamma,A}$  is the sum of the multiplicities of all the 
roots of $P$ greater than $\unit$,
we obtain, from Corollary~\ref{carac-mul},
that $P_{\ell}=\eval{P}(\unit)=\unit\neq \zero$ for
$\ell=m'_{\gamma,A}$ and $\ell=m'_{\gamma,A}+m_{\gamma,A}$.
Then, for generic values of $a$, we have $p_{\ell}\neq 0$
for $\ell=m'_{\gamma,A}$ and $\ell=m'_{\gamma,A}+m_{\gamma,A}$,
which implies that the polynomial
$p^{(\unit)}$ of~\eqref{defpc} is non zero and that its valuation is equal
to $m'_{\gamma,A}$ and its degree is equal to $m'_{\gamma,A}+m_{\gamma,A}$.
Then, for generic values of $a$,  $m_\gamma= m_{\gamma,A}$ and
$m'_\gamma=m'_{\gamma,A}$, and so $m''_\gamma=m''_{\gamma,A}$, 
which finishes the
proof of Point~\eqref{th-1-finite} of  Theorem~\ref{th-1-gen} 
in the case $G=\sat(\eval A(\gamma),U,V)$.

To prove the same results when $G=\opt(\eval A(\gamma))$, we shall
use the following lemma.

\begin{lemma}
Let $b$ be a matrix polynomial with coefficients in $\C^{n\times n}$ and degree $d$,
let $B\in\rmin^{n\times n}$ be a min-plus matrix such that $\perm B\neq\zero$,
and let $(U,V)$  be a  Hungarian pair with respect to $B$.
Then, the matrix polynomials $b^{G}:=b_0^G+\cdots + \iY^d b_d^G$ defined
respectively with $G=\sat(B,U,V)$ and with $G=\opt(B)$
have the same eigenvalues.
\end{lemma}
\begin{proof}
Let $\sigma$ be an optimal permutation for $B$.
Then, by Corollary~\ref{coro3}, the identity map is an optimal permutation for 
$C:=\Psigmainvmin B$, $(U_{\sigma^{-1}},V)$ is  a Hungarian pair with respect to $C$, and we have $\opt(C)=\opt(B)_{\sigma^{-1},\idp}$,
and $\sat(C,U_{\sigma^{-1}},V)=\sat(B,U,V)_{\sigma^{-1},\idp}$.
Then, by Corollary~\ref{coro5},
$\opt(C)$ is 
 the disjoint union of the strongly connected components
of $\sat(C,U_{\sigma^{-1}},V)$.
Multiplying the matrix polynomials $b$ and $b^G$ on the left by $P^{\sigma^{-1}}$,
and using the above notation, 
we deduce that $P^{\sigma^{-1}} b^{G}=c^{G'}$,
with $c=P^{\sigma^{-1}}b$ and $G'=G_{\sigma^{-1},\idp}$.
Since $P^{\sigma^{-1}}$ is invertible, the sequence of eigenvalues
of $b^G$ is the same as the one of  $P^{\sigma^{-1}} b^{G}$,
hence the same as the one of $c^{G'}$.
If $G'$ has several strongly connected components,
then, up to a permutation, the matrix polynomial
$c^{G'}$ is bloc triangular, with diagonal blocs
$c^{S}$, for each strongly connected component $S$ of $G'$.
Then, the sequence of eigenvalues of $c^{G'}$
is the disjoint union of the sequences of eigenvalues of the
matrix polynomials  $c^S$.
Since the sequence of eigenvalues
of $b^G$ is the same as the one of $c^{G'}$,
and, by the above properties, the strongly connected components of 
$G'=G_{\sigma^{-1},\idp}$ are the same for  $G=\opt(B)$ and $G=\sat(B,U,V)$,
we deduce that the sequence of eigenvalues of $b^G$ are the same
for $G=\opt(B)$ and for $G=\sat(B,U,V)$.
\end{proof}
Applying this result to the matrix polynomial 
$b= a_0^{\G_{0}(A,\gamma)}+\cdots +\iY^d a_d^{\G_{d}(A,\gamma)}$, 
and the min-plus matrix $B=\eval A(\gamma)$,
we get that the matrix polynomials $a^{(\gamma)}$ given in~\eqref{pencilsat},
with $G_k=\G_{k}(A,\gamma)\cap G$ for all $0\leq k\leq d$,
and respectively $G=\sat(\eval A(\gamma),U,V)$ or $G=\opt(\eval A(\gamma))$
have same eigenvalues.
So the assertion of Point~\eqref{th-1-finite} of  Theorem~\ref{th-1} (resp.\ 
Theorem~\ref{th-1-gen}) for $G=\sat(\eval A(\gamma),U,V)$ 
is equivalent to the one for $G=\opt(\eval A(\gamma))$,
which finishes the proof of this point.

\def\cprime{$'$}

\end{document}